%% file: OTBasis.tex
\documentclass[11pt]{article}
\usepackage[utf8]{inputenc}
\usepackage[T1]{fontenc}

\usepackage{enumerate}
\usepackage{url}
\usepackage{epsf}
\usepackage{epsfig}
\usepackage{graphicx}
\usepackage{verbatim}
\usepackage{leftidx}

\usepackage{color}

\usepackage{amsmath,amsfonts,amssymb,amsthm}
\usepackage{wasysym}
\usepackage[retainorgcmds]{IEEEtrantools}

\usepackage{stmaryrd}

\usepackage{ gensymb }
\usepackage{ esint }

\usepackage{hyperref}

\input{\pathCommands/CalligraphicLetters.tex}

\input{\pathCommands/Operators.tex}

\input{\pathCommands/TheoremsEng.tex}

\DeclareMathOperator\Cone{Cone}

\DeclareMathOperator\Hull{Hull}

\renewcommand\subset{\subseteq}

\makeatletter
\newcommand{\oset}[2]{%
  {\mathop{#2}\limits^{\vbox to -.5\ex@{\kern-\tw@\ex@
   \hbox{\scriptsize #1}\vss}}}}
\makeatother

\DeclareMathOperator\ConeS{{\oset {$\circ$} {\Cone}}} 

\usepackage{algorithm,algpseudocode,algorithmicx}

\def\RZ2{\cF(\Z^2)}

\def\FO{{\text{\tiny WS}}}
\def\SB{{\text{\tiny LBR}}}
\def\FD{{\rm FD}}

\def\huefg{h(\Delta^+_e u(x), \Delta^+_f u(x), \Delta^+_g u(x))}

\author{
Jean-David Benamou$^*$, 
Francis Collino%
\footnote{%
Mokaplan, INRIA, Domaine de Voluceau BP 105 78153, Le Chesnay Cedex France
}, 
Jean-Marie Mirebeau%
\footnote{%
CNRS, University Paris Dauphine, UMR 7534, Laboratory CEREMADE, Paris, France.\newline 
\emph{E-mail}: mirebeau<at>ceremade.dauphine.fr, jean-david.benamou<at>inria.fr \newline
}
}

\def\cVin{\cV_\Omega}
\def\cVu{{\overline \cV}}

\def\cDu{\overline \cD}
\def\DE2{DE2}
\def\FX{{\mathbb U}} 
\def\imagetop#1{\vtop{\null\hbox{#1}}}

\usepackage{pgffor}

\begin{document}
\title{
Monotone and Consistent discretization\\ of the  
Monge-Ampere operator
}
\maketitle
\date{}

\begin{abstract} 
We introduce a novel discretization of the Monge-Ampere operator, simultaneously consistent and degenerate elliptic, hence accurate and robust in applications. 
These properties are achieved by exploiting the arithmetic structure of the discrete domain, assumed to be a two dimensional cartesian grid. 
The construction of our scheme is simple, but its analysis relies on original tools seldom encountered in numerical analysis, such as the geometry of two dimensional lattices, and an arithmetic structure called the Stern-Brocot tree.
Numerical experiments illustrate the method's efficiency.
\end{abstract}

\section{Introduction}
We introduce a new discretization of the Monge-Ampere operator, on two dimensional cartesian grids, which is consistent and preserves at the discrete level a fundamental property of the continuous operator: degenerate ellipticity.
Discrete degenerate ellipticity \cite{Oberman:2006bd} implies strong guarantees for the numerical scheme: a comparison principle, convergence of discrete solutions towards the continuous one in the setting of viscosity solutions, and convergence of Euler iterative solvers for the discrete system \cite{Oberman:2006bd}. 
Some Degenerate Elliptic (DE) schemes for the Monge-Ampere (MA) Partial Differential Equation (PDE) already exist \cite{Froese:2011ka,Oberman:2006bd}, but they suffer from several flaws: they are strongly non-local, and only approximately consistent. 
Consistent \emph{non DE} schemes such as \cite{Loeper:2005fn,Brenner:2012ka} offer better accuracy, but require the PDE solution to be sufficiently smooth and the discrete numerical solver to be well initialized. Filtered schemes \cite{Froese:2013ez} nonlinearly combine several existing schemes, in order to cumulate their advantages (here degenerate ellipticity and consistency), or mitigate their defects. Their definition and their analysis are however complex, and their application requires to adjust several parameters.  For a recent overview of the numerical  approaches to solving 
the  Monge-Amp\`ere equation, see Glowinski, Feng and Neilan \cite{Feng:2013cg}. 

We introduce a new numerical scheme, Monge-Amp\`ere using Lattice Basis Reduction (MA-LBR), which is both consistent\footnote{Assuming the solution hessian condition number is uniformly bounded} \emph{and} degenerate elliptic. Lattice Basis reduction is a tool from discrete geometry, which arises here due to the interaction of the cartesian discretization grid, with the anisotropic nature of the Monge-Ampere operator. This operator is indeed invariant under all linear changes of variables with unit determinant, unlike e.g.\ the Laplacian which is merely invariant under orthogonal transformations. 
The MA-LBR belongs is inspired by the Wide-Stencil \cite{Oberman:2006bd} family of schemes. 
Using another arithmetic tool, the Stern-Brocot tree, we solve a second issue plaguing these methods (in addition to consistency errors): our discretization stencil needs not be chosen a priori (which usually involves a difficult arbitrage between scheme locality, consistency error and available CPU time), but  can be generated automatically in a guaranteed, parameter free and solution adapted manner.
Numerical experiments \S \ref{sec:Num} illustrate the MA-LBR accuracy and robustness.

We fix throughout this paper a convex open bounded 
domain $\Omega \subset \R^2$. Given a density $\rho \in C^0(\overline \Omega, \R_+^*)$, and some Dirichlet data $\sigma \in C^0(\partial \Omega, \R)$, we set the goal of approximating numerically the unique viscosity solution \cite{Crandall:1992kn,Gutierrez:2001wq} of
\begin{equation}
\label{eq:MAD}
\begin{cases}
\det(\nabla^2 u) = \rho & \text{on } \Omega, \\
u = \sigma & \text{on } \partial \Omega,\\
u  \text{ convex.}
\end{cases}
\end{equation}
Our framework admittedly does not encompass solutions of the weaker Alexandrov type, where $\rho$ is merely a non-negative measure. 
If $\Omega$ is convex but not strictly convex, then the Dirichlet data $\sigma$ is assumed to be convex on any segment of $\partial \Omega$. Let us point out that optimal transport, from $\Omega$ to another domain $\Omega'$, equipped with densities $\rho$, $\rho'$, admits a PDE formulation similar in spirit to \eqref{eq:MAD}: $\rho'(\nabla u) \det (\nabla^2 u) = \rho$, $\nabla u(\Omega) \subset \Omega'$, $u$ convex. The gradient non-linearity, and the second boundary condition, raise difficulties  \cite{Urbas:1996uz,Benamou:2014jw} that we choose not to address in the present paper, focusing instead on the Monge Ampere operator $\det(\nabla^2 u)$. 

We assume that the PDE domain $\Omega$ is discretized on a cartesian grid: $\Omega \cap h R (\xi + \Z^2)$, where $h>0$ is the grid scale, $R$ is an arbitrary rotation, and $\xi$ is an offset. For notational simplicity, and up to a linear change of coordinates, we limit our attention to the canonical values of these parameters,
so that the discrete domain is
\begin{equation*}
X := \Omega \cap \Z^2.
\end{equation*}
\begin{definition}
We denote by $\FX$ the collection of discrete maps $u : X \cup \partial \Omega \to \R$. A (discrete) operator is a map $\cD : \FX \to \R^X$. It associates to each $u \in \FX$ a collection of values $\cD u(x)$, $x \in X$.
\end{definition}
The notations $\cD u$ and $\cD(u)$ refer to the same object, which is a map $X \to \R$, and are used interchangeably with the aim of improving readability.
In numerical experiments, the values of $u \in \FX$ on $X$ are the unknowns, while the values on $\partial \Omega$ are the supplied boundary data:  $u_{|\partial \Omega} = \sigma$. 
For each $e \in \Z^2$ we introduce a second order differences operator $\Delta_e$, built so that $\Delta_e u(x) \approx  \< e, (\nabla^2 u(x)) e\>$, 
where $u \in \FX$ and $x \in X$, and where $\<\cdot, \cdot\>$ denotes the euclidean scalar product on $\R^2$.
In the simplest case where $x\pm e \in \Omega$, we set  
\begin{equation}
\label{eqdef:Delta}
\Delta_e u (x) := u(x+e)-2 u(x)+u(x-e).
\end{equation}
When $x \in X$ is close to $\partial \Omega$, the points $x+e$ or $x-e$ may not belong to $\Omega$. Denoting by $h^\pm$ the only element of $]0,1]$ such that $x\pm h^\pm e \in X \cup \partial \Omega$, we define 
\begin{equation}
\label{eqdef:DeltaBoundary}
\Delta_e u (x) := \frac 2 {h^++h^-} \left(\frac {u(x+h^+ e)-u(x)}{h^+}+ \frac{u(x-h^{-} e)-u(x)}{h^-}\right). 
\end{equation}
Let us again point out that if $h^+ < 1$, then the value $u(x+h^+ e)$ is the supplied boundary data $\sigma(x+h^+ e)$. On the other hand if $h^+ = h^- = 1$, then \eqref{eqdef:Delta} and \eqref{eqdef:DeltaBoundary} coincide.
No other consistent approximation of $\<e, (\nabla^2 u(x)) e\>$ can be built using the values $u(x+h^+e)$, $u(x)$ and $u(x-h^- e)$. 

Discretizations of the Monge-Ampere operator $\det(\nabla^2 u)$ are typically built upon the operators $\Delta_e$. Consider for instance the Finite Differences (FD) discretization \cite{Loeper:2005fn}
\begin{equation}
\label{eqdef:DFD}
\cD^\FD := \Delta_{(1,0)} \Delta_{(0,1)} - (\Delta_{(1,1)}-\Delta_{(1,-1)})^2/16.
\end{equation}
Given such a discrete operator $\cD$, the discrete analog of \eqref{eq:MAD} takes the form:
\begin{equation}
\label{eq:DiscreteSys}
\text{Find $u \in \FX$, such that $\cD u = \rho$ on $X$, and $u_{|\partial \Omega} = \sigma$.}
\end{equation}
This discrete system lacks a counterpart of the constraint of convexity in \eqref{eq:MAD} because (i) there is no unique notion of discrete convexity but several competing approaches, see for instance \cite{Mirebeau:Vn5Iu9VK,Oberman:2011wy}, and (ii) some form of discrete convexity constraint can often be embedded in the equation $\cD u = \rho$, see \S \ref{sec:HierarchyIntro}.
From a theoretical and a practical standpoint, choosing $\cD^\FD$ in \eqref{eq:DiscreteSys} is a risky bet: second order convergence can often be observed in numerical experiments, see \S \ref{sec:Num}, but only on rather easy cases and with a good initialization for the numerical solver. 
Robustness results (existence, uniqueness, and algorithmic guarantees) are limited to discretizations obeying an additional property: a counterpart of the ellipticity of the (opposite of the) Monge-Ampere operator $-\det (\nabla^2 u)$. 

We use the notion of discrete degenerate ellipticity \cite{Oberman:2006bd}, slightly specialized due to our focus on MA. Degenerate Elliptic  Monge-Ampere numerical schemes cannot be strictly local, unlike \eqref{eqdef:DFD}, but instead need to take into account some long range second order differences, indexed by a possibly wide stencil. 

\begin{definition}
\label{def:Stencil}
A stencil is a finite set $V \subset \Z^2\sm \{0\}$ which is symmetric with respect to the origin (i.e. $-e \in V$ for each $e \in V$).
\end{definition}

\begin{definition} 
(DE2 scheme)
A numerical scheme $-\cD$ is Degenerate Elliptic, with stencil $V$, iff for each $x \in X$ the quantity $\cD u (x)$ is a non-decreasing, locally Lipschitz function of the second order differences $\Delta_e u(x)$, $e \in V$.
\end{definition}

\begin{figure}
\centering
\includegraphics[width=3.8cm]{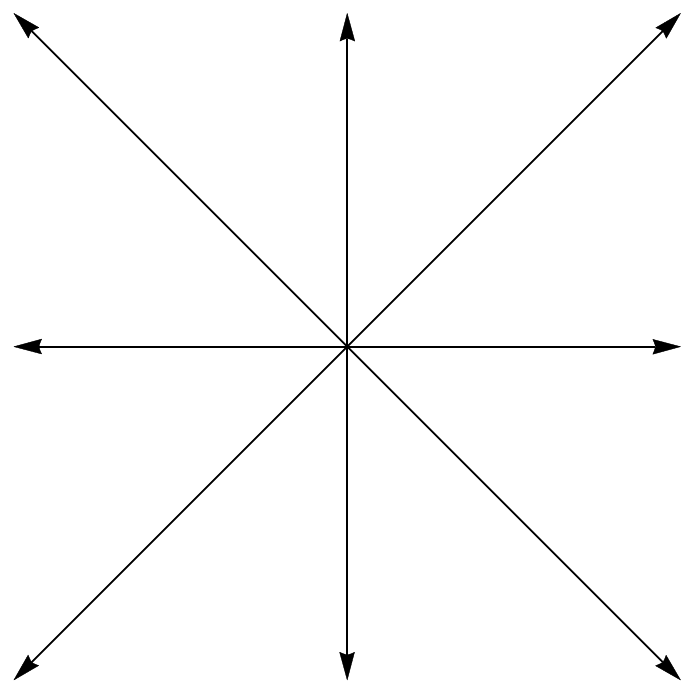}
\includegraphics[width=3.8cm]{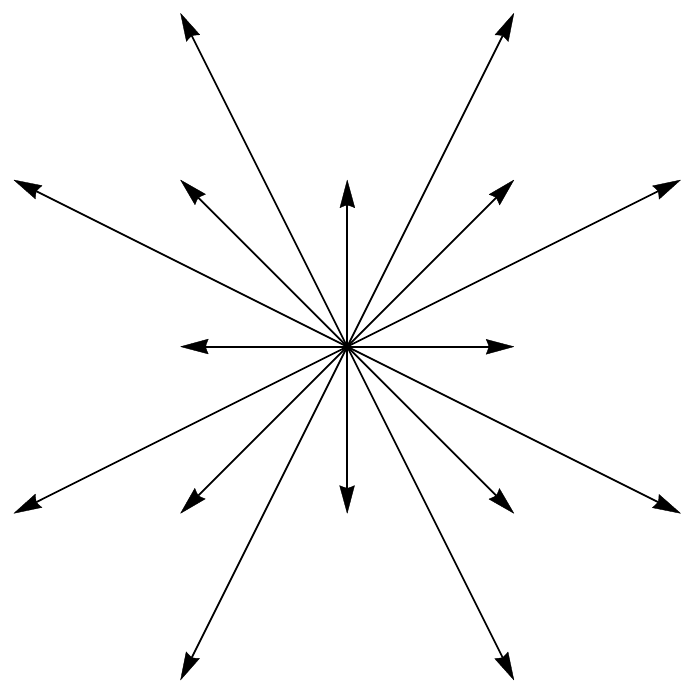}
\includegraphics[width=3.8cm]{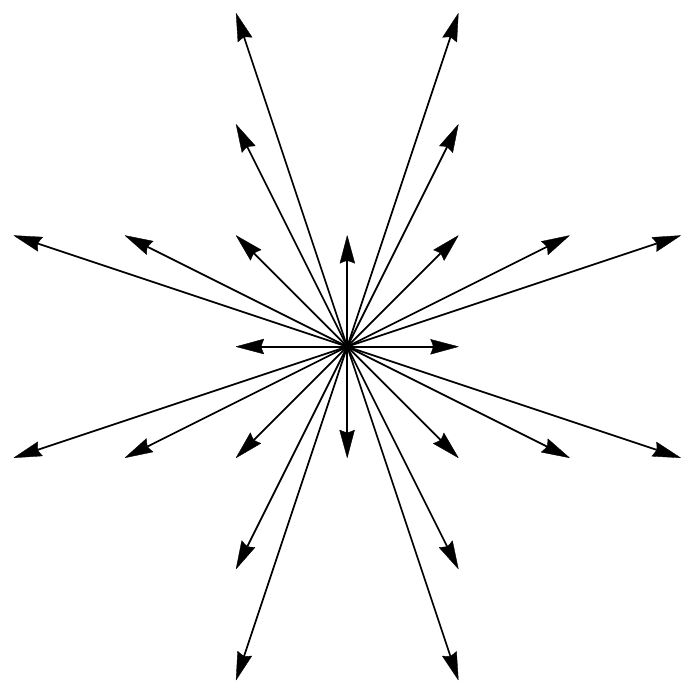}
\includegraphics[width=3.8cm]{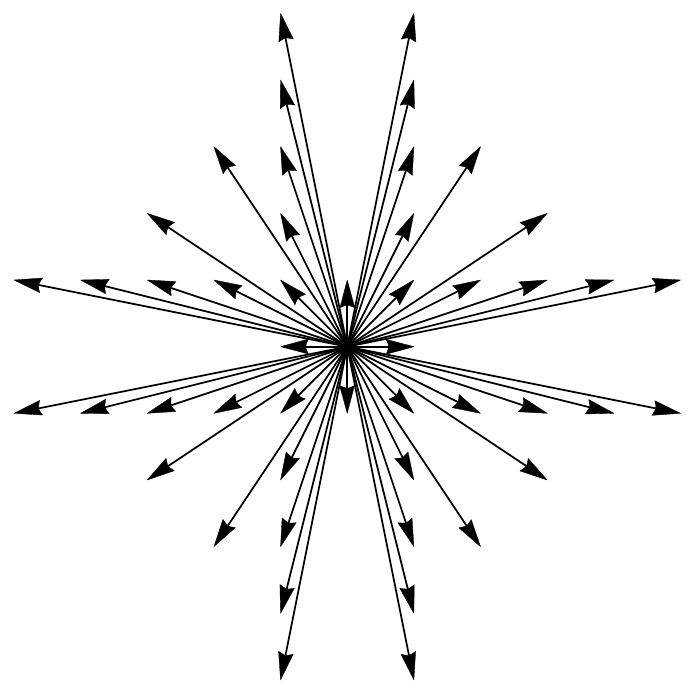}
\caption{
Examples of stencils $V \subset \Z^2$, containing $8$, $16$, $24$ and $48$ elements respectively. Wider stencils yield smaller consistency errors, see Figure \ref{fig:RelativeError}.
}
\label{fig:Stencils}
\end{figure}

Observing that the second order difference $\Delta_e u (x)$ can be expressed as a non-negative weighted sum of first order differences \eqref{eqdef:DeltaBoundary}, we immediately find that 
a DE2 scheme is degenerate elliptic in the sense of \cite{Oberman:2006bd}. 
DE2 schemes are  also  positive difference operators in the sense of \cite{Kuo:1992hf} in this paper schemes are indeed  built using directional 
second order finite differences. 
In particular, for any $\ve>0$, the slightly perturbed operator $-\cD_\ve$, defined by $\cD_\ve u(x) := \cD u(x) - \ve u(x)$, is \emph{proper} degenerate elliptic \cite{Froese:2013ez}. This in turn implies that the discrete system \eqref{eq:DiscreteSys} associated with $\cD_\ve$ has a unique solution, which can be computed with a geometric convergence rate using an iterative Euler scheme. We refer to \cite{Froese:2013ez} and references therein for these results and will say no more on this analytic machinery in the rest of the paper, focusing instead on the algebraic structure of discrete Monge-Ampere operators.

Froese and Oberman \cite{Froese:2011ka} numerically address the MA PDE using a DE2 operator, referred to as the Wide Stencil (WS) scheme. Given a stencil $V$, and denoting $\Delta_e^+ := \max \{0, \Delta_e\}$:
\begin{equation}
\label{eq:FO}
\cD^\FO_V u (x):= \min_{\substack{(f,g) \in V^2\\\text{orthogonal}}} \frac{\Delta^+_{f} u (x)}{\|f\|^2} \,  \frac{\Delta^+_{g} u (x)}{\|g\|^2}.
\end{equation}
The minimum is taken over all pairs of vectors $(f,g) \in V^2$ which are orthogonal, in the sense that $\<f,g\>=0$. For instance $(1,0), (0,1)$, or $(2,1), (-1,2)$.
We introduce a variant of this operator, which does not rely on pairs of orthogonal stencil vectors, but on superbases of the lattice $\mZ^2$. 
\begin{definition}
A basis of $\Z^2$ is a pair $(f,g)\in (\Z^2)^2$ such that $|\det(f,g)| = 1$.\\
A superbase of $\mZ^2$ is a triplet $(e,f,g) \in (\mZ^2)^3$ such that $e+f+g=0$, and $(f,g)$ is a basis of $\mZ^2$.
\end{definition}

\begin{figure}
\centering
\includegraphics[width=2.2cm]{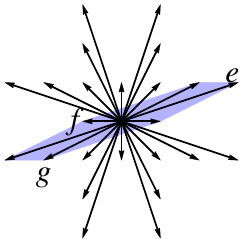}
\includegraphics[width=2.2cm]{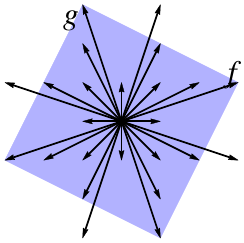}
\hspace{1.5cm}
\includegraphics[width=2cm]{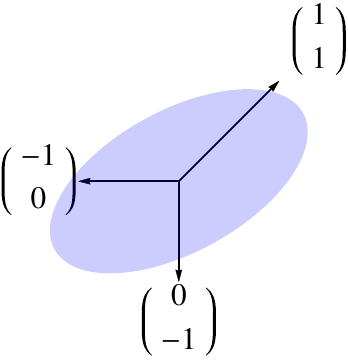}
\includegraphics[width=2.5cm]{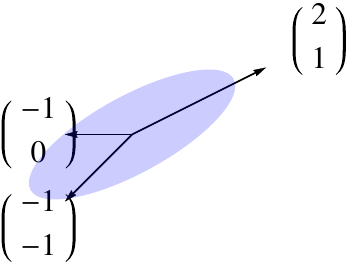}
\hspace{-0.2cm}
\includegraphics[width=3cm]{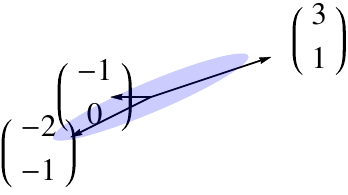}
\caption{Left: A Stencil $V$, a superbase $(e,f,g)\in V^3$, an orthogonal pair $(f,g) \in V^2$. Right: $M$-obtuse superbase, see Definition \ref{def:ObtuseSuperbase}, and ellipse $\{v \in \mR^2; \<v, M v\>\leq 1\}$ for some $M \in S_2^+$.}
\end{figure}

The MA-LBR operator, associated to a stencil $V$, is defined by 
\begin{equation}
\label{eq:SB}
\cD^\SB_V u (x) :=
\min_{\substack{(e,f,g) \in V^3\\ \text{superbase} } } \huefg 
\end{equation}
where for $a,b,c \in \mR_+$ we define
\begin{equation}
\label{eqdef:h}
h(a,b,c) := 
\begin{cases}
b c  \text{ if } a \geq b+c, \text{ and likewise permuting } a,b,c,\\
\frac 1 2 (ab+bc+ca) - \frac 1 4(a^2+b^2+c^2)  \text{ otherwise}.
\end{cases}
\end{equation}
Remark \ref{rem:Geom} provides a geometric interpretation for these at first abstruse formulas.
The operators \eqref{eq:FO} and \eqref{eq:SB} are DE2 since the product $\times : \R_+^2 \to \R_+$ is non-decreasing in each variable, as well as the function $h : \R_+^3 \to \R_+$, see Lemma \ref{lem:hDiff}.

\paragraph{Outline.}
We discuss in \S \ref{sec:LocalIntro} the consistency of the MA-LBR, and show in particular that a finite stencil is sufficient to achieve consistency for all quadratic functions of condition number below a given bound.
For more simplicity and efficiency we introduce in \S \ref{sec:HierarchyIntro} an automatic stencil construction for the MA-LBR, which is adaptive, local, anisotropic, parameter free, and has good consistency guarantees. 
The proofs of the results appearing in \S \ref{sec:LocalIntro} and \S \ref{sec:HierarchyIntro} are postponed to \S \ref{sec:Local} and \S \ref{sec:Hierarchy} respectively.

\paragraph{Notations.} For each $e=(a,b) \in \mR^2$ we denote $e^\perp := (-b,a)$. If $e=(a,b) \in \Z^2$ then $\gcd(e) := \gcd(a,b)$. Given  pairwise distinct $x_1,x_2,x_3 \in \R^2$, we denote by $[x_1,x_2]$ the segment of endpoints $x_1,x_2$, and by $[x_1, x_2, x_3]$ the triangle of vertices $x_1,x_2,x_3$.

\subsection{Consistency}
\label{sec:LocalIntro}

\begin{figure}
\centering
\begin{tabular}{ccccc}
\imagetop{\includegraphics[width=3.7cm]{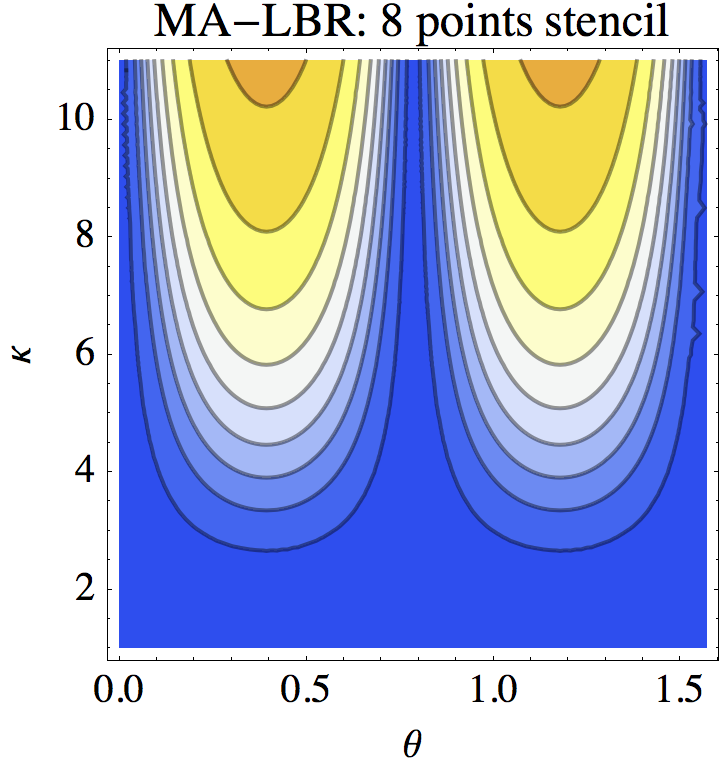}} &
\imagetop{\includegraphics[width=3.7cm]{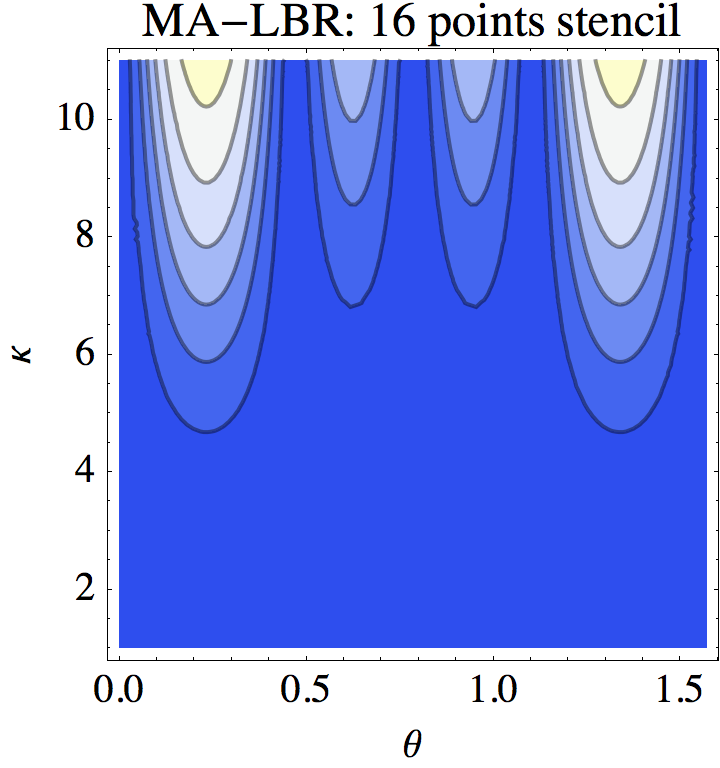}} &
\imagetop{\includegraphics[width=3.7cm]{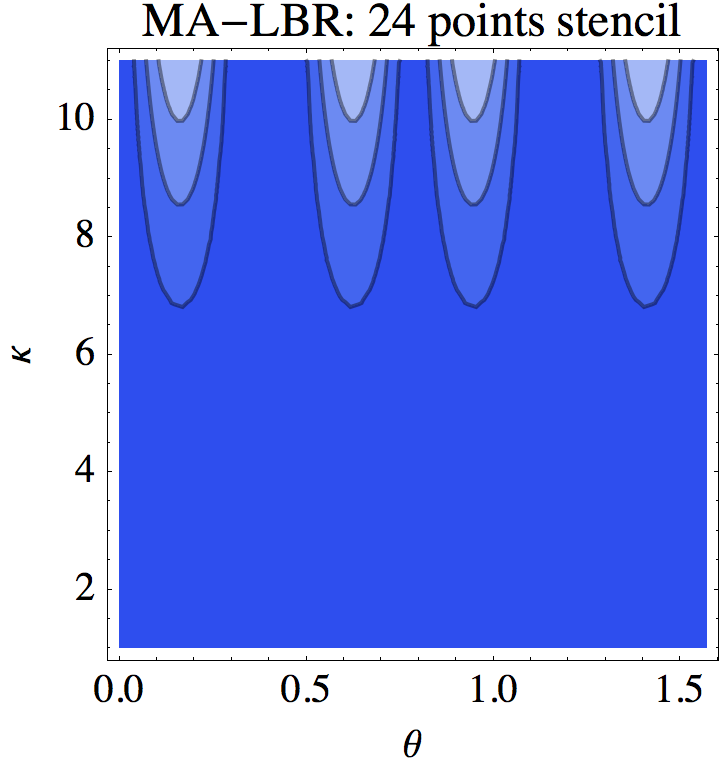}} &
\imagetop{\includegraphics[width=3.7cm]{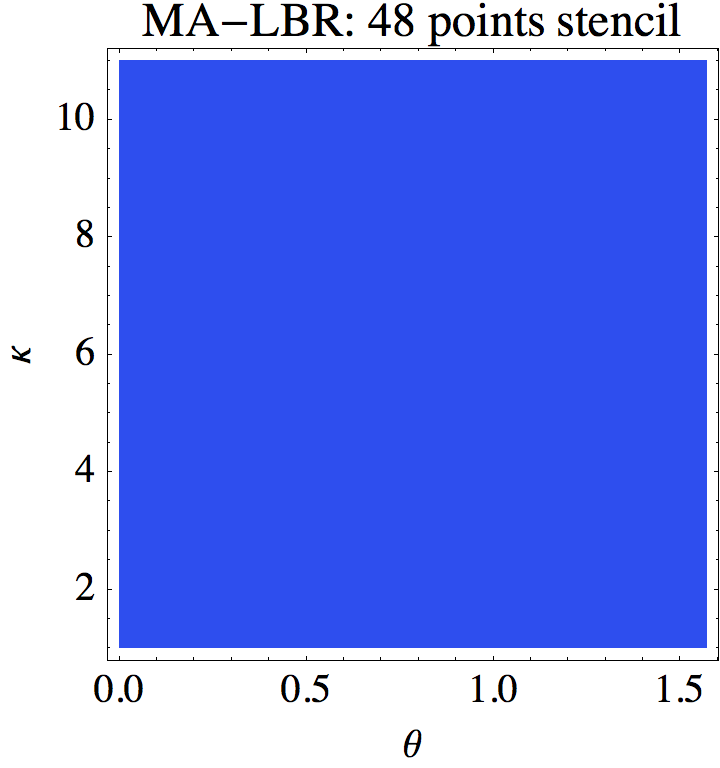}} &
\imagetop{\includegraphics[width=0.7cm]{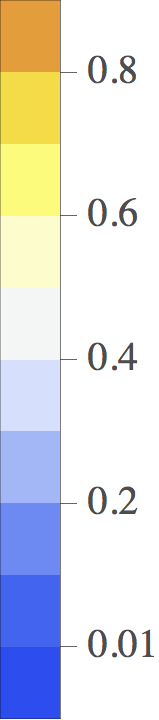}} \\
\imagetop{\includegraphics[width=3.7cm]{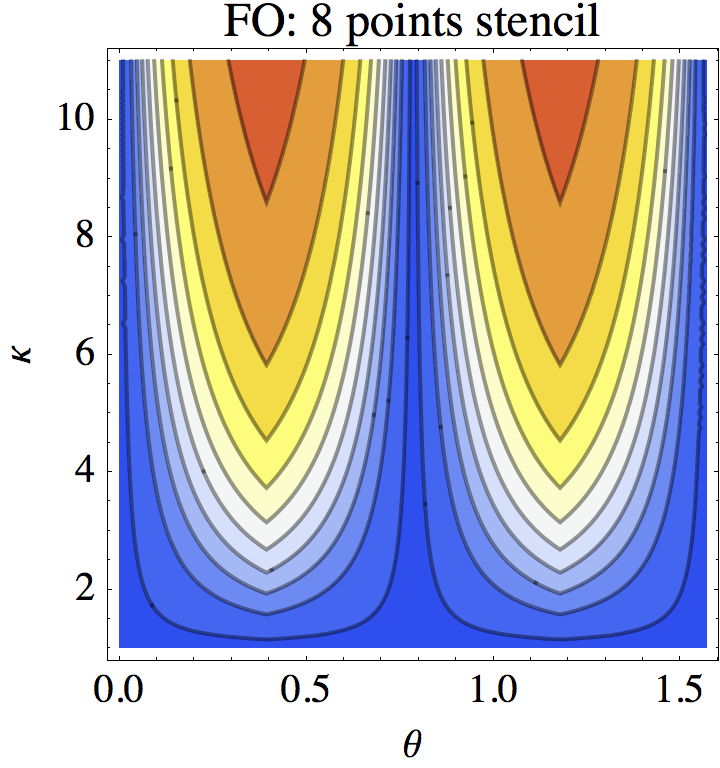}} & 
\imagetop{\includegraphics[width=3.7cm]{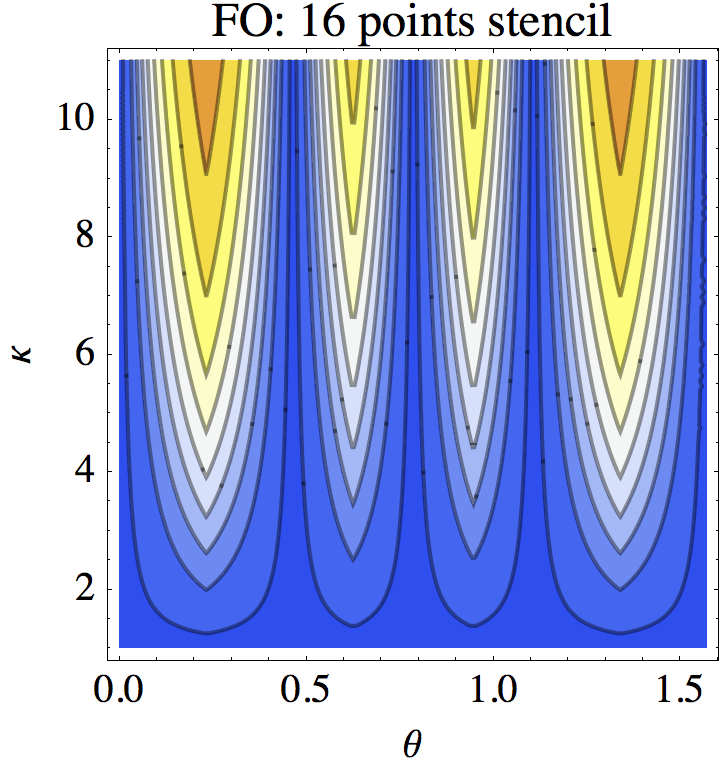}} &
\imagetop{\includegraphics[width=3.7cm]{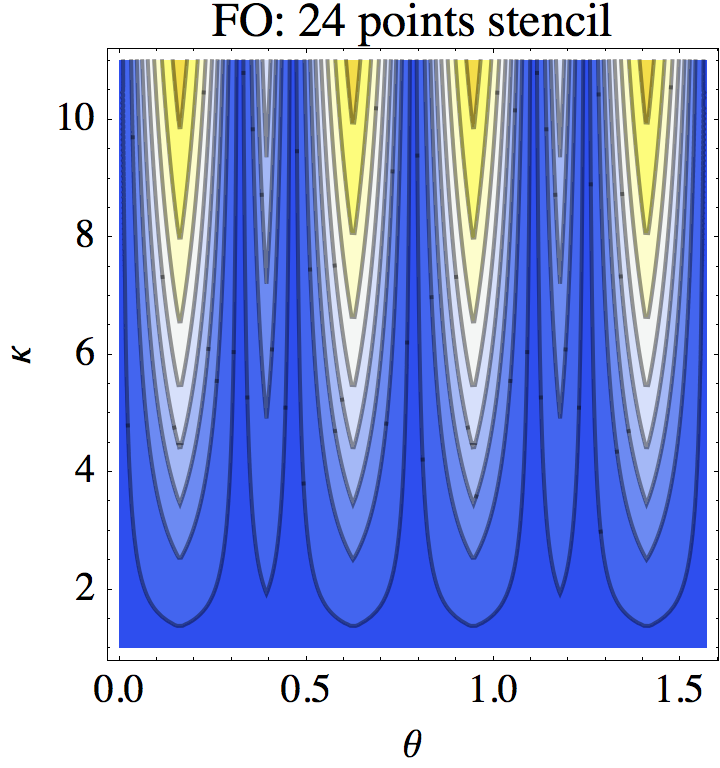}} &
\imagetop{\includegraphics[width=3.7cm]{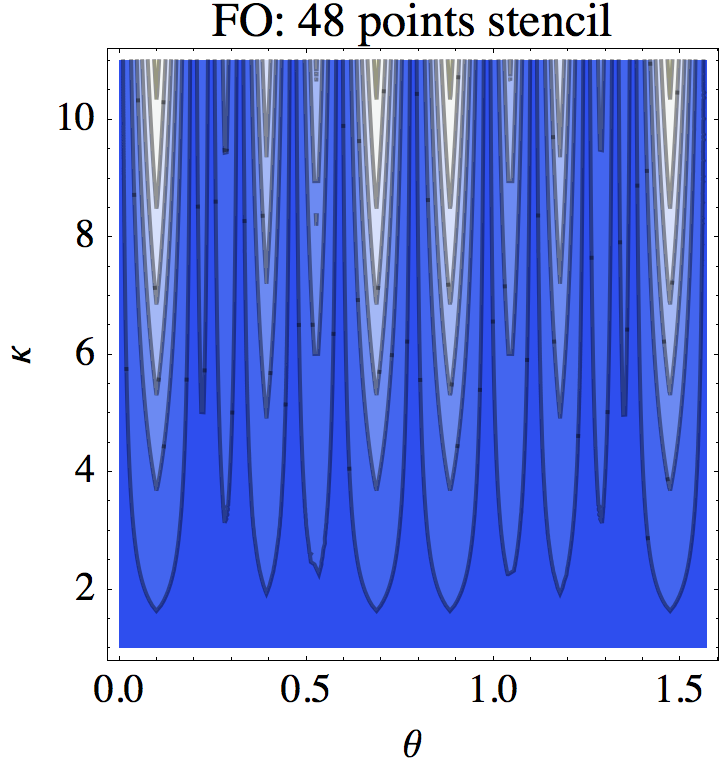}} &
\imagetop{\includegraphics[width=0.7cm]{Images/Consistency/Legend.png}}
\end{tabular}
\caption{Relative consistency error $(\cD(u_M) - \det(M))/\cD(u_M)$ for quadratic functions \eqref{eqdef:UM}, with the schemes MA-LBR (top) and WS (bottom), using the stencils of Figure \ref{fig:Stencils}. 
See \eqref{eq:MatParam} for the parametrization $M(\kappa, \theta)$ of symmetric matrices, by their condition number $\kappa^2$ and their orientation $\theta$.
Note that the MA-LBR consistency error vanishes for a large set of matrices.
}
\label{fig:RelativeError}
\end{figure}

The consistency analysis of the numerical schemes FD, FO and the MA-LBR reveals significant differences. 
We denote by $S_2$ the collection of symmetric matrices of size $2 \times 2$, and by $S_2^+$ those which are positive definite. For each $M \in S_2^+$ we introduce a quadratic map $u_M \in \FX$, defined by $u_M(x) := \<x, Mx\>/2$, $x \in X \cup \partial \Omega$. Since the second order difference operator $\Delta_e$ is consistent, for any $e \in \Z^2$, it is exact for $u_M$. Summarizing one has 
\begin{align}
\label{eqdef:UM}
u_M(x) &:= \frac 1 2 \<x, M x\>, &
\Delta_e u_M(x) &= \<e, M e\>.
\end{align}
\begin{definition}
The consistency set of an operator $\cD$ is the collection of matrices $M \in S_2^+$ for which $\cD(u_M) = \det(M)$, identically on $X$.
\end{definition}

One easily checks that the consistency set of the finite differences discretization $\cD^\FD$, see \eqref{eqdef:DFD}, is the whole $S_2^+$. In fact the identity $\cD^\FD(u_M) = \det(M)$ also holds for non-definite matrices $M \in S_2$, although they are irrelevant for our application. Since that scheme is not DE, this consistency does not imply convergence results. 
As illustrated in Figures \ref{fig:RelativeError} and \ref{fig:Consistency}, schemes WS and MA-LBR have in contrast non-trivial consistency sets, depending on the chosen stencil. Matrices $M \in S_2^+$ are parameterized in these figures by their condition number $\kappa^2 \in [1, \infty[$, and the orientation $\theta \in [0, \pi]$ of their first eigenvector $e_\theta$:
\begin{equation}
\label{eq:MatParam}
M(\kappa, \theta)  = \kappa^{-1} \, e_\theta \otimes e_\theta + \kappa \, e_\theta^\perp \otimes e_\theta^\perp, \qquad \text{with } e_\theta = (\cos \theta, \sin \theta).
\end{equation}

The consistency analysis of $\cD^\FO_V$ is based on Hadamard's theorem \cite{Froese:2011ka}: for all $M \in S_2^+$, and any pair $(f,g) \in (\R^2)^2$ of non-zero orthogonal vectors, one has $\<f,M f\> \<g , M g\> \geq \|f\|^2 \|g\|^2 \det(M)$, with equality iff $f$ and $g$ are eigenvectors of $M$.
As a result, scheme $\cD^\FO_V$ is only consistent on a negligible subset of $S_2^+$: those matrices which eigenvectors lie in $V$, see Figures \ref{fig:RelativeError} and \ref{fig:Consistency}. 
From a theoretical standpoint, convergence results are obtained in \cite{Froese:2011ka} by increasing the stencil size, up to infinity, as the discretization grid scale tends to zero. In practical cases, finding the optimal stencil size is non-trivial, see \S \ref{sec:Num}.

\begin{figure} 
\centering
\begin{tabular}{ccc}
\imagetop{\includegraphics[width=4cm]{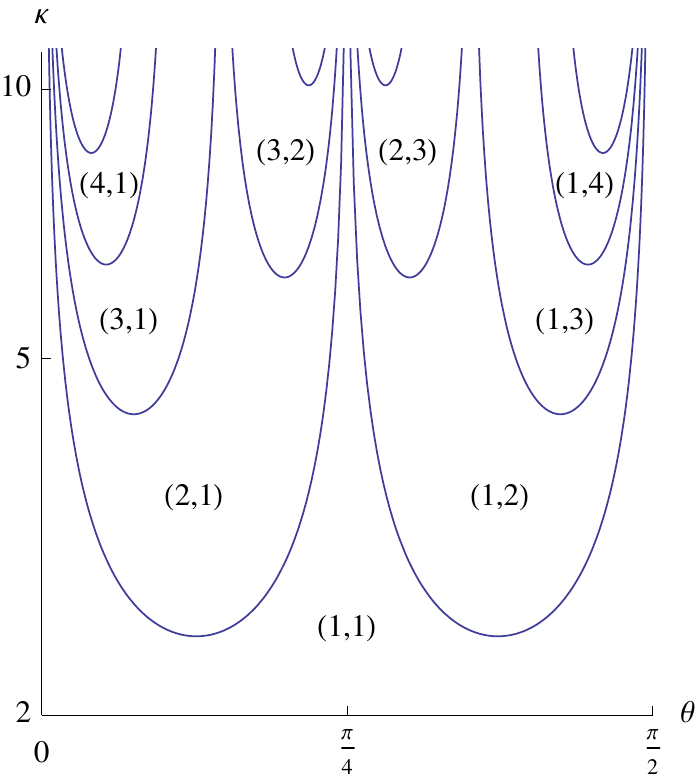}}
&
\imagetop{\includegraphics[width=4cm]{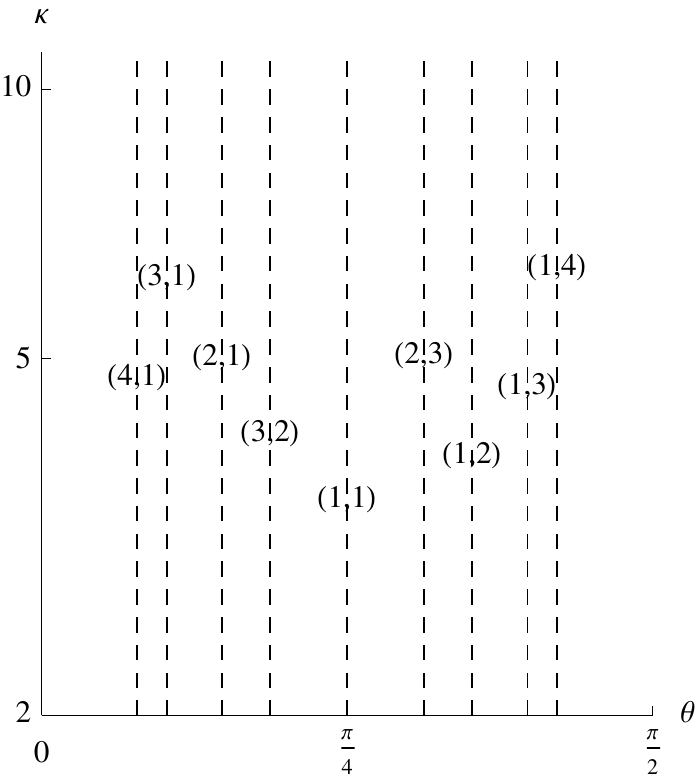}}
&
\imagetop{
\begin{tabular}{c|cc}
\includegraphics[width=1.8cm]{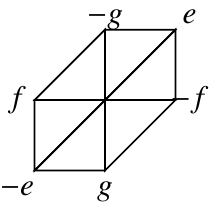}   &
\includegraphics[width=1.8cm]{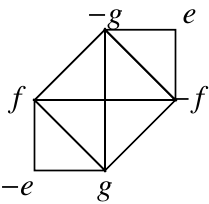} &
\includegraphics[width=1.8cm]{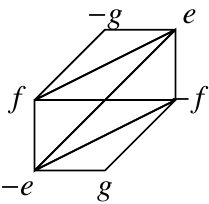} 
\\
\includegraphics[width=1.8cm]{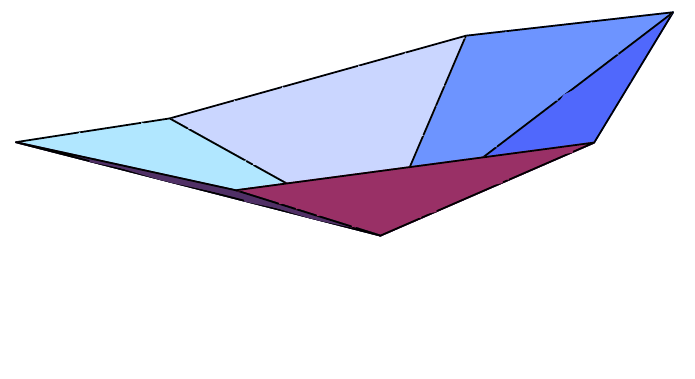} &
\includegraphics[width=1.8cm]{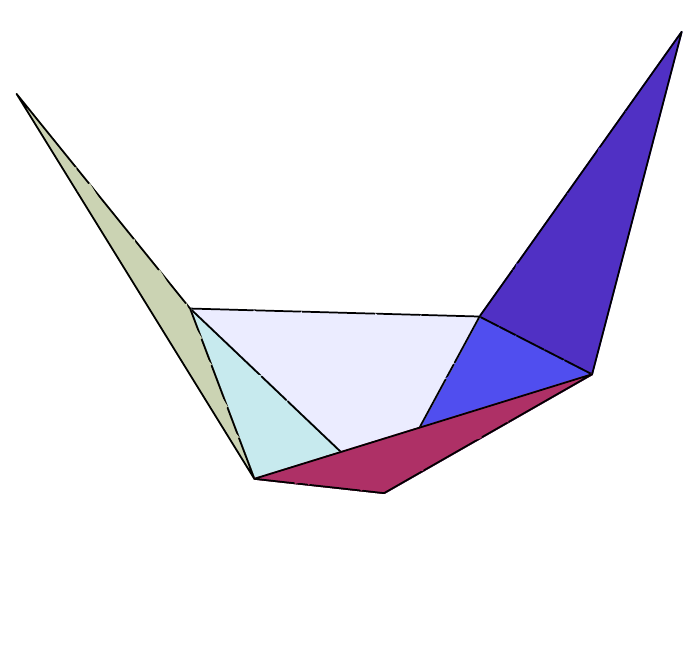} &
\includegraphics[width=1.8cm]{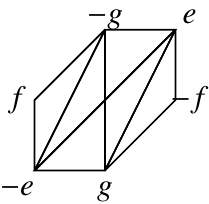} 
\end{tabular}
}
\end{tabular}
\caption{
Left: element $e$ of largest euclidean norm of an $M(\kappa, \theta)$-obtuse superbase.
Center left: an eigenvector $e$ of $M(\kappa, \theta)$. The consistency set of $\cD^\SB_V$ (resp.\ $\cD^\FO_V$) is the union of the regions (resp.\ the dashed lines) corresponding to all stencil elements $e \in V$.
Right: triangulations mentioned in Remark \ref{rem:Geom}, and some 3D visualisations of the maximal convex extension $U$. 
}
\label{fig:Consistency}
\label{fig:Subgradient}
\end{figure}

The key concept in the MA-LBR consistency analysis is the notion of $M$-obtuse superbase, which originates from lattice geometry \cite{Conway:1992cq} 
(a lattice is a discrete subgroup of $\R^n$ containing a basis, such as $\mZ^n$). 
It was already applied to PDE discretizations in \cite{mirebeau2012anisotropic,Fehrenbach:2013ut}.

\begin{definition}
\label{def:ObtuseSuperbase}
Let $M \in S_2^+$. A superbase $(e_0,e_1,e_2)$ of $\mZ^2$ is said $M$-obtuse iff $\<e_i , M e_j\> \leq 0$ for all $0 \leq i < j \leq 2$.
\end{definition}

\begin{theorem}[Consistency]
\label{th:Consistency}
A matrix $M \in S_2^+$ is in the consistency set of $\cD^\SB_V$ iff  there exists $(e,f,g) \in V^3$ which form an $M$-obtuse superbase. 
\end{theorem}

The following remark attempts to give a geometrical interpretation of the function \eqref{eqdef:h} and of Theorem \ref{th:Consistency}. 
The results of this section, Theorem \ref{th:Consistency}, Remark \ref{rem:Geom} and Theorem \ref{th:Minimality}, are established in \S \ref{sec:Local}.

\begin{remark}[Geometric interpretation]
\label{rem:Geom}
Let $M \in S_2^+$, and let $(e,f,g)$ be a superbase of $\mZ^2$. Let $U$ be the maximal convex map bounded above by $u_M$ at the points $x,x\pm e, x\pm f, x\pm g$. Then 
$\huefg = \Area(\partial U(x))$
(the Lebesgue measure of the subgradient of $U$ at $x$, which is a natural relaxation of the Monge-Ampere operator \cite{Gutierrez:2001wq}). The map $U$ is polygonal, on one of the four triangulations illustrated Figure \ref{fig:Subgradient}. The identity $|\partial U(x)| = \det(M)$ holds for the first triangulation only, which corresponds to an $M$-obtuse superbase $(e,f,g) $.
\end{remark}

Strikingly, one cannot hope for a DE2 scheme more localized than the MA-LBR. Finding well localized numerical schemes, involving small stencils, is a natural objective \cite{Kocan:1995iq}.
\begin{theorem}[Minimality]
\label{th:Minimality}
Let $\cD$ be a DE2 scheme with stencil $V$.
If the consistency set of $\cD$ contains the neighborhood of a matrix $M\in S_2^+$, then there exists $(e,f,g) \in \Hull(V)^3$ which form an $M$-obtuse superbase.
\end{theorem}

The following algorithm and proposition, dating back to Selling \cite{Selling:1874uo,Conway:1992cq}, constructively shows the existence of an $M$-obtuse superbase for each $M \in S_2^+$, without which Theorems \ref{th:Consistency} and \ref{th:Minimality} would be mostly vacuous. It is worth noting that this algorithm extends to dimension three \cite{Conway:1992cq}. Proposition \ref{prop:convalg2} also immediately implies that all matrices $M \in S_2^+$ with condition number $\|M\| \|M^{-1}\| \leq \kappa^2$ are simultaneously  in the consistency set of the MA-LBR operator $\cD^\SB_V$ with stencil 
\begin{equation}
\label{eq:StencilKappa}
V := \{ e \in \mZ^2;\,  \gcd(e)=1, \|e\| \leq 2 \kappa\}.
\end{equation}

\begin{algorithm}
\caption{Construction of an $M$-obtuse superbase (Selling \cite{Conway:1992cq}).}
\label{algo:Selling}
\begin{tabular}{lll}
 \multicolumn{3}{l}{\textbf{Initialize} $e_0 \leftarrow (-1,-1)$, \ $e_1 \leftarrow (1,0)$, \ $e_2 \leftarrow (0,1)$. 
 (Or any other initial superbase.)}\\
 \multicolumn{3}{l}{\textbf{While} the superbase $(e_0, e_1, e_2)$ is not $M$-obtuse \textbf{do}}\\
& \multicolumn{2}{l}{Find $0\leq i < j \leq 2$ such that $\<e_i, M e_j \> > 0$, and set $(e_0, e_1, e_2) \leftarrow (e_i-e_j,\ e_j,\  -e_i)$.}
\end{tabular}
\end{algorithm}

In order to analyse this algorithm, we associate to each $M \in S_2^+$ the norm 
\begin{equation}
\|e\|_M := \sqrt{\<e, M e\>}, \quad  e \in \mR^2
\end{equation}
\begin{proposition}[Existence of an $M$-obtuse superbase, Selling 1874]
\label{prop:convalg2}
Algorithm 2 terminates, and the final state of $(e_0,e_1,e_2)$ is an $M$-obtuse superbase. Furthermore $\|e_i\|^2 \leq 2 \|M\| \|M^{-1}\|$ for each $0 \leq i \leq 2$.
\end{proposition}

\begin{proof}
To each superbase associate the energy $\cE(e_0, e_1, e_2) := \|e_0\|_M^2+\|e_1\|_M^2+ \|e_2\|^2_M$. One easily checks that $\cE(e_i-e_j,\, e_j,\, -e_i) = \cE(e_0, e_1, e_2) - 4 \<e_i, M e_j\>$, for any $0 \leq i < j \leq 2$.
Denoting by 
$(e_0^n, e_1^n, e_2^n)$ 
the successive superbases generated in Algorithm \ref{algo:Selling}, we observe that the energies $\cE(e_0^n, e_1^n, e_2^n)$ are strictly decreasing by construction. Noticing that there exists only a finite number of superbases with energy below a given bound, we find that the algorithm terminates. At termination, the continuation criterion ``the superbase $(e_0,e_1,e_2)$ is not $M$-obtuse'' is false, which establishes the first point.

Let $\cF(e_0,e_1,e_2) := \|e_0\|^2 + \|e_1\|^2 + \|e_2\|^2$, and let $n$ be the number of loop iterations. Then 
\begin{equation}
\|M^{-1}\|^{-1}\cF(e_0^n, e_1^n, e_2^n) \leq \cE(e_0^n, e_1^n, e_2^n) \leq \cE(e_0^0, e_1^0, e_2^0) \leq \|M\| \cF(e_0^0, e_1^0, e_2^0) = 4 \|M\|,
\end{equation}
which immediately implies the announced bound on the obtuse superbase elements norm. 
\end{proof} 

The MA-LBR consistency error is typically smaller than with the WS scheme, for a given stencil $V$, see Figure \ref{fig:Consistency}. Furthermore while the WS consistency is an asymptotic property, depending on the stencil angular resolution, the MA-LBR has in contrast a consistency set of non-empty interior, and its elements can be identified with a simple test, see Theorem \ref{th:Consistency}.
Unfortunately, choosing the MA-LBR effective stencil $V$ before a numerical simulation remains at this point a puzzle for the practitioner. The option \eqref{eq:StencilKappa} is not practical because: (a) Uniform bounds $\kappa^2$ on the hessian matrix $\nabla^2 u$ condition number  of solutions to \eqref{eq:MAD}, are seldom available. (b) Even if such a bound $\kappa$ is available, the set \eqref{eq:StencilKappa} can be quite large, with cardinality $\gtrsim \kappa^2$. This becomes an issue if the bound is pessimistic, or if the solution hessian $\nabla^2 u$ does degenerate in some places, such as along the domain boundary $\partial \Omega$.  A third issue (c) is that there is no clear way to a-posteriori validate the choice of a given stencil: would the numerical solution be improved with a larger one ?


Selling's algorithm, in contrast with the inefficiency of \eqref{eq:StencilKappa}, adaptively produces an $M$-obtuse superbase in only few iterations. 
We present in the next section an adaptive, anisotropic, parameter free and guaranteed  stencil refinement algorithm, which eliminates the implementation difficulties (a), (b), (c) above. Under the hood, it amounts to an adaptation of Selling's algorithm to non-quadratic functions, see \S \ref{sec:PropagationOnTrees}.

\subsection{Hierarchical stencil refinement}
\label{sec:HierarchyIntro}

The previous section fully characterized the consistency set of the MA-LBR operator $\cD^\SB_V$, associated to a stencil $V$. Larger stencils provide consistency on larger collections of matrices, as established in Theorem \ref{th:Consistency}, and illustrated on Figure \ref{fig:Consistency}. 
Excessively large stencils are however unpractical,
since the CPU cost of evaluating the MA-LBR operator \eqref{eq:SB} is proportional to their cardinality. 
Adapting Selling's obtuse superbase construction, Algorithm \ref{algo:Selling}, 
we show that one can emulate an MA-LBR with extremely large stencils for a limited numerical cost. 

Our adaptive variant of the MA-LBR operator is defined by Algorithm \ref{algo:D2p} below, which is $6$ lines long and only involves elementary operations. Its analysis (and Definition \ref{def:Stencils} of a mild structural constraint on stencils) relies on an arithmetic construction named the Stern-Brocot tree, already used in \cite{Bonnans:2004ud,Mirebeau:2012wm} for the discretization of anisotropic PDEs. Definitions \ref{def:OPlus}, \ref{def:GraphT} introduce this structure. Propositions \ref{prop:Decomp} and \ref{prop:V8T} are variants of commonly known facts on the Stern-Brocot tree which proof is, for completeness, presented in the appendix. 

\begin{proposition} 
\label{prop:Decomp}
The identity $e=f+g$ defines a one to one correspondance between:
\begin{itemize}
\item Vectors $e = (a,b) \in \Z^2$, such that $\gcd(a,b)=1$ and $a b \neq 0$.
\item Direct acute bases $(f,g)$ of $\Z^2$ (i.e. $(f,g) \in (\Z^2)^2$, $\det(f,g) = 1$ and $\<f,g\> \geq 0$).
\end{itemize}
\end{proposition}

\begin{definition}
\label{def:OPlus}
We emphasize the unique dÃ©composition introduced in Proposition \ref{prop:Decomp} by using the notation 
$e = f \oplus g.$
Whenever we write $e = f \oplus g$, we implicitly limit our attention to vectors $e$ satisfying the assumptions of Proposition \ref{prop:Decomp}. 
\end{definition}

For instance $(7,5) = (3,2) \oplus (4,3)$, $(3,2) = (2,1)\oplus (1,1)$, and $(1,1) = (1,0) \oplus (0,1)$. 
If $e=f\oplus g$, then $(e,-f,-g)$ is a superbase of $\mZ^2$; all superbases happen to be of that form, up to a permutation of their elements, see Lemma \ref{lem:SuperbaseDecomp}.
The next proposition shows how to generate numerous decompositions of the form of Proposition \ref{prop:Decomp}.
\begin{proposition}
\label{prop:SumChildren}
If $e=f\oplus g$, then $f+e = f\oplus e$ and $e+g = e \oplus g$.
\end{proposition}
\begin{proof}
We check $\det(f,e) = \det(f,f+g)=\det(f,g)=1$, and $\<f,e\> = \<f,f+g\> = \|f\|^2+\<f,g\> \geq 0$. Hence $(f,e)$ is a direct acute basis of $\Z^2$. Likewise for $(e,g)$.
\end{proof}

\begin{definition}
\label{def:GraphT}
We introduce a graph $\mT$, with vertices $\{e \in \Z^2; \, \gcd(e) = 1\}$, and edges $e \to f \oplus e$ and $e \to e \oplus g$ for each $e = f\oplus g$.
\end{definition}

We say that an edge $e \to e'$ of $\mT$ leaves from $e$ and arrives at $e'$.
We denote by $V_8$ the eight point stencil illustrated on Figure \ref{fig:Stencils} (left).
\begin{equation}
\label{eqdef:V8}
V_8 := \{(a,b) \in \{-1,0,1\}^2; \, a b \neq 0\}.
\end{equation}

\begin{proposition}
\label{prop:V8T}
The set $V_8$ has one element in each connected component of $\mT$.
The four points of the form $(\pm 1, 0)$ or $(0, \pm 1)$ are isolated. The four other points are the root of complete infinite binary trees, each one entirely contained in a quadrant of the plane. 
\end{proposition}

\begin{figure}
\centering
\includegraphics[width=4.5cm]{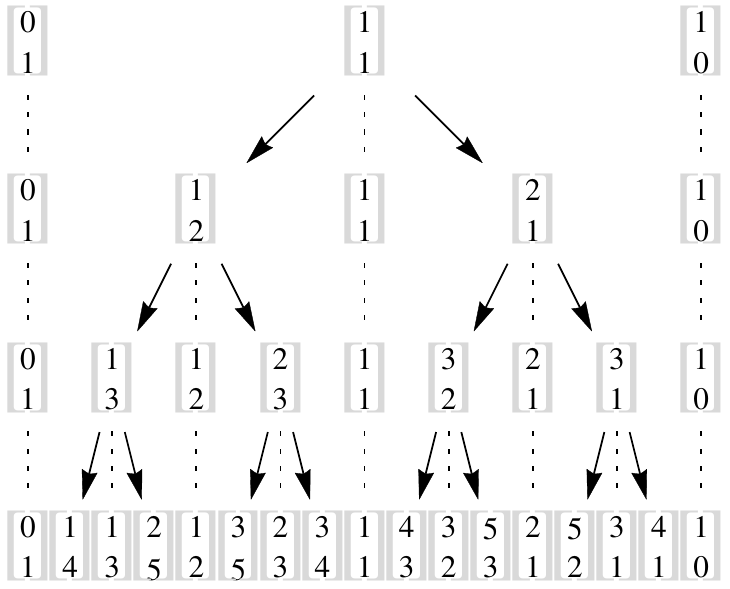} 
\hspace{0.1cm}
\includegraphics[width=3cm]{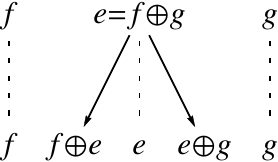} 
\hspace{0.1cm}
\includegraphics[width=3.8cm]{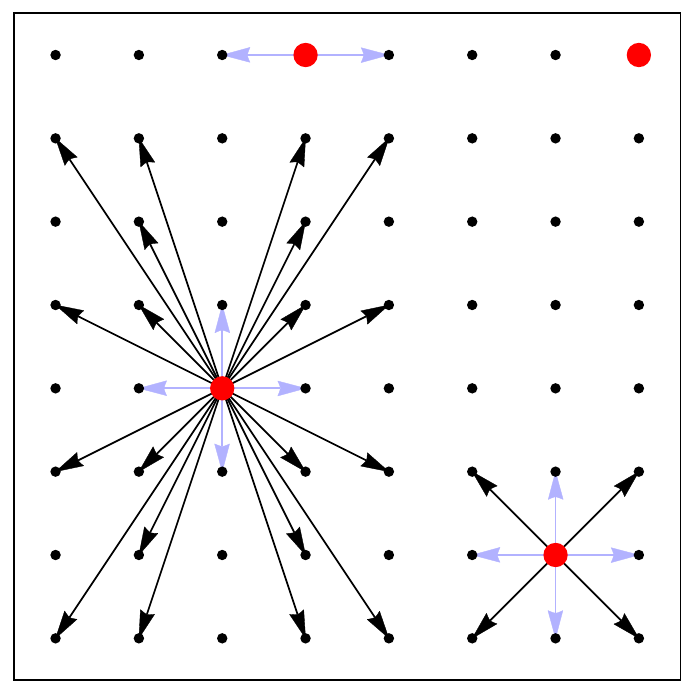} 
\includegraphics[width=3.8cm]{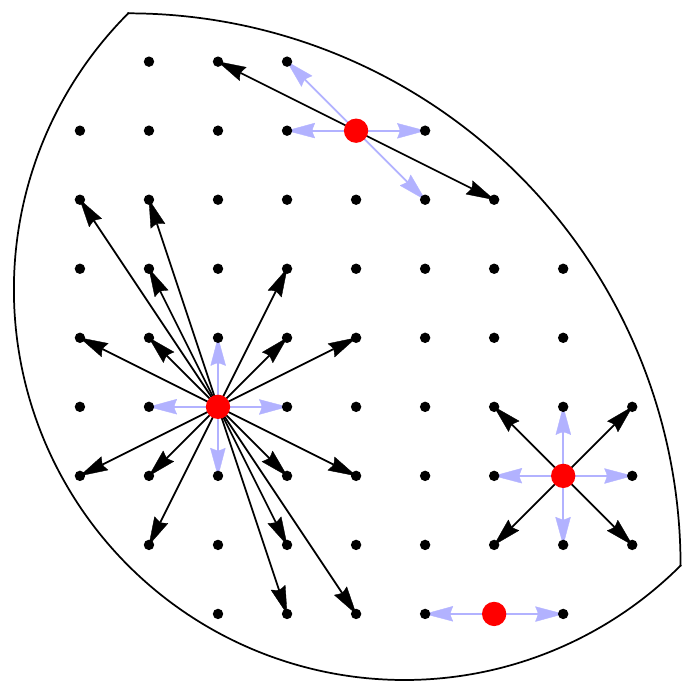}    
\caption{
Stern Brocot binary tree (left), and its local structure (center left). 
Right: some convex domains, and the associated discretization grids. Plain black arrows: set $\cV_\Omega(x)$, for some $x \in X$. Light blue arrows: vectors $e \in \mT \sm \cVin(x)$ such that $x \pm e \in \Omega$, as in property (Reachability) of stencils.
}
\label{fig:Vin}
\label{fig:SternBrocotTree}
\end{figure}


The Stern-Brocot tree is the subgraph corresponding to the first quadrant, with vertices $\mT^+ := \{e = (a,b) \in \mT;\, a>0,\, b>0\}$, see Figure \ref{fig:SternBrocotTree}.
This complete infinite binary tree originates from arithmetic, and in the literature a vertex $e = (a,b) \in \mT^+$ is often identified with the positive irreducible fraction $a/b$. 

The MA-LBR adaptive variant, presented below, requires stencils with a special structure. For each $x \in X$ we introduce the set 
\begin{equation}
\label{eqdef:Vin}
\cV_\Omega(x) := \{ e = f \oplus g; \, x \pm e, x \pm f, x \pm g \in \Omega\},
\end{equation}
where ``$x \pm e \in \Omega$'' stands for ``$x+e \in \Omega$ and $x-e \in \Omega$''.
Note that the continuous domain $\Omega$ could in \eqref{eqdef:Vin} be replaced with the discrete one $X := \Omega \cap \Z^2$: indeed for any $z \in \Z^2$, one has $z \in \Omega$ iff $z \in X$.
Any subset $V$ of $\mT$ is regarded as a subgraph of $\mT$, equipped with all edges of $\mT$ having their endpoints in $V$.
\begin{definition}
\label{def:Stencils}
A family of stencils $\cV$ is the data, for each $x \in X$, of a stencil $\cV(x)$ satisfying $V_8 \subset \cV(x) \subset \mT$, and the additional structural properties:
\begin{itemize}
\item (Hierarchy) The set $V_8$ has an element in each connected component of $\cV(x)$.
\item (Reachability) $\cV(x)$ contains each $e \in \mT \sm \cVin(x)$ such that $x \pm e \in \Omega$.
\end{itemize}
\end{definition}
Practical recommendations regarding the construction of stencils are discussed after Theorem \ref{th:Hierarchy}.
These structural constraints are in practice not hard to satisfy. 
Condition (Hierarchy) is natural in view of the Stern-Brocot tree structure, and is satisfied by all stencils illustrated on Figure \ref{fig:Stencils}.
Condition (Reachability) ensures that the stencils can be extended, in the sense of Proposition \ref{prop:VExtended} below. It is vacuous for all $x \in X$ at distance $\geq 1$ from $\partial \Omega$ in general, and entirely vacuous in the case of a box domain, see Proposition \ref{prop:ReachabilityVacuous} and Corollary \ref{corol:ReachabilityVacuousBox}. 


The sets $\cV_\Omega(x)$, $x \in X$, are small or empty when $x$ is close to $\partial \Omega$, but typically huge when $x$ is far from $\Omega$, see Figure \ref{fig:Vin}. 
They \emph{do not} constitute a family of stencils, but they can be used to extend an existing family of stencils, as in the next definition. 
\begin{definition}
To each family of stencils $\cV$, we associate the family of sets $\cVu$ defined by $\cVu(x) :=  \cV(x) \cup \cVin(x)$, $x \in X$.
\end{definition}

\begin{proposition}[Extension of stencils]
\label{prop:VExtended}
If $\cV$ is a family of stencils, then $\cVu$ also is. 
\end{proposition}

We next introduce the MA-LBR operator $\cD_\cV$ associated to a family $\cV$ of stencils, as well as a hierarchical variant $\cDu_\cV$. (The expected $^\SB$ superscript is omitted for readability.)
\begin{equation}
\label{eqdef:DVux}
\cD_\cV u (x) := 
\min_{\substack{(e,f,g) \in \cV(x)^3\\ \text{superbase} } } \huefg.
\end{equation}

\begin{algorithm}
\caption{Hierarchical operator $\cDu_\cV u(x)$ (the final value of $\mD$).}
\label{algo:D2p}
\begin{tabular*}{\textwidth}{@{\extracolsep{\fill} } l r}
\textbf{Initialize} a variable $f \leftarrow (1,0)$, and list $G \leftarrow [(0,1), (-1,0)]$. Set also $\mD \leftarrow + \infty.$\\
\textbf{While} $G$ is non-empty \textbf{do}\\
\phantom{bla} Denote by $g$ the first element of $G$, and set $e := f+g$. \\
\phantom{bla} \textbf{If} $e \in \cV(x)$, or [$e \in \cVin(x)$ and $\Delta_e u (x) < \Delta_f u (x)+ \Delta_g u (x)$] & \hspace{-1cm}(Refinement test)\\
\phantom{blabla} \textbf{then} prepend $e$ to $G$, and set $\mD \leftarrow \min\{\mD, \, \huefg \}$ \\ 
\phantom{blabla} \textbf{else} remove $g$ from $G$ and set $f \leftarrow g$.
\end{tabular*}
\end{algorithm}

Our main result Theorem \ref{th:Hierarchy} states that the MA-LBR operator $\cD_\cVu$ associated to the large stencils $\cVu$ coincides in all cases of interest with the hierarchical, adaptive variant $\cDu_\cV$. Algorithm \ref{algo:D2p} amounts to a depth-first transversal of a finite subtree of the Stern-Brocot tree, see \S \ref{sec:Algorithm} and \cite{Mirebeau:2012wm} where a similar approach is used for the discretization of Hamilton-Jacobi PDEs. This subtree is characterized by the stopping criterion (Refinement test),
allowing to reject useless branches of $\mT$ where the minimum \eqref{eqdef:DVux} defining $\cD_\cVu u (x)$ cannot be attained. In the case of a quadratic map $u_M$, $M \in S_2^+$, Algorithm \ref{algo:D2p} explores a single branch of the Stern-Brocot tree, just as Selling's algorithm, see \S \ref{sec:PropagationOnTrees}.

We say that a property holds ``on $X$'' iff it holds at each point of $X$.

\begin{theorem}[Adaptive pruning equals extensive sweeping]
\label{th:Hierarchy}
Let $\cV$ be a family of stencils, and let $u \in \FX$. 
If $\cDu_\cV u >0$ on $X$, or $\cD_\cVu u >0$ on $X$, then we have $\cDu_\cV u=  \cD_\cVu u$ on $X$.
\end{theorem}
In our experiments \S \ref{sec:Num} with $\cDu_\cV$, we use reasonably large stencils $\cV(x)$ on a layer of a few grid points along $\partial \Omega$, where $\cVin(x)$ is small or empty, see Figure \ref{fig:Vin}. We use in contrast the minimal stencils $V_8$ elsewhere since they are adaptively completed by Algorithm \ref{algo:D2p}.

The identity $\cDu_\cV u=  \cD_\cVu u$ may break down when these two operators vanish at some points of $X$. 
This is fortunately not an issue since (i) the problem of interest \eqref{eq:DiscreteSys} has by assumption a positive right hand side, and (ii), the positivity of the MA-LBR operator is equivalent to the positivity of second order differences, see Proposition \ref{prop:PositiveDifferences} below, which is a natural discrete counterpart of the convexity constraint present in the original Monge-Ampere problem \eqref{eq:MAD}.

Consider a smooth function $U : \Omega \to \R$, and a point $x \in \Omega$ such that $\det(\nabla^2 U (x)) > 0$. Then $U$ is convex (or concave) on a neighborhood of $x$. 
The next proposition establishes a discrete analog of this property.
Consider $u \in \FX$ and a family $\cV$ of stencils. The discrete counterpart of $\det(\nabla^2 U(x)) > 0$ is $\cD_\cV u(x) > 0$, while the counterpart of the convexity of $U$ locally around $x$ is the positivity of the second order differences centered at $x$: $\Delta_e u(x)>0$, $e \in \cV(x)$. 

\begin{proposition}[Discrete convexity]
\label{prop:PositiveDifferences}
Let $u \in \FX$, let $\cV$ be a family of stencils on $X$, and let $x \in X$. Then 
\begin{equation*}
\cD_\cV u(x) > 0 \quad \Longleftrightarrow \quad  \forall e \in \cV(x), \, \Delta_e u(x) > 0.
\end{equation*}
\end{proposition}

Oberman \cite{Oberman:2011wy} numerically addressed variational problems posed on the cone of convex functions by imposing the positivity of second order differences, $\Delta_e u(x)$ for all points $x \in X$ and all vectors $e$ within some given stencil $V$. It is also known, see Appendix A of \cite{Mirebeau:Vn5Iu9VK}, that any discrete map $u : X \to \R$ satisfying $\Delta_e u(x)\geq 0$ whenever $x, x\pm e \in X$, needs to coincide $u_{|X'} = U_{|X'}$ with a global convex function $U : \Omega \to \R$ on the subsampled grid $X' := X \cap (2 \Z^2)$ of points with even coordinates.

\begin{remark}
Adaptivity in PDE discretizations often refers to the context where a sequence $u_n$ of discrete maps is generated along an iterative procedure, as well as a sequence $\cD_n$ of operators, and $\cD_{n+1}$ depends on $u_n$. Our understanding in this paper is different: there is no underlying iteration, but a single operator $\cD_\cVu$ which is evaluated in a subtle and cheap way as $\cDu_\cV$.
\end{remark}

\input{OTBasis_ConsistencyMinimality.tex}

\input{OTBasis_HierarchicalStencilRefinement.tex}

\input{OTBasis_NumericalExperiments.tex}

\section{Conclusion}

The MA-LBR introduced in this paper is a new numerical scheme for two dimensional Monge-Ampere PDEs, which combines consistency and degenerate ellipticity. In our numerical experiments, these properties become accuracy and robustness. 
Our scheme is not strictly local and may involve long range stencils, but they are built in a sparse, adaptive, and anisotropic manner using a guaranteed and parameter free refinement algorithm. Our construction is also shown to be as local as it can be, among symmetric, consistent and degenerate elliptic schemes for the Monge-Ampere PDE.
The analysis of our algorithm involves tools seldom used in the context of numerical analysis, including elements of lattice geometry \cite{Conway:1992cq}, and the arithmetic of the Stern-Broccot tree.

Future research will be devoted to some natural questions that the present method cannot directly address. In particular (i) the computation of solutions of the weaker Alexandroff type, (ii) the additional difficulties tied to the discretization of optimal transport problems, instead of boundary value problems, (iii) Monge-Ampere problems posed on three dimensional domains, and (iv) local adaptation and refinement of the discretization grid.\\

\textbf{Acknowledgements:} The  authors gratefully acknowledge the support of the ANR, through the project ISOTACE (ANR-12-MONU-0013), the grant NS-LBR ANR-13-JS01-0003-01, and INRIA through the "action exploratoire" MOKAPLAN.   

\appendix 
\section{Structure of the Stern-Brocot tree}

\subsection{Unique decomposition $e=f\oplus g$}
\label{sec:DecompIntegerVector}

The two following propositions together establish Proposition \ref{prop:Decomp}.

\begin{proposition}
Let $f,g$ be a direct acute basis of $\mZ^2$. Then $e := f+g$ has co-prime  coordinates, both non-zero.
\end{proposition}

\begin{proof}
One has $\det(f,e) = \det(f,f+g) = 1$, hence the coordinates of $e$ are co-prime as announced. Also $\|e\|^2 = \|f\|^2+ 2\<f,g\>+\|g\|^2 \geq 2$, hence $\|e\|>1$. Assuming for contradiction that a coordinate of $e$ is zero, we find that the other one can only be $\pm 1$, since they are co-prime. But then $\|e\|=1$ which is a contradiction. This concludes the proof.
\end{proof}

In the following, a quadrant of the plane is a set of the form: for some $\alpha, \beta \in \{-1,1\}$
\begin{equation*}
Q_{\alpha, \beta} := \{(a,b) \in \R^2; \, \alpha a \geq 0, \, \beta b \geq 0\}.
\end{equation*}

\begin{proposition}
\label{prop:Quadrant}
Let $e = (a,b) \in \Z^2$ be such that $\gcd(a,b) = 1$ and $a b \neq 0$. Then there exists a unique direct basis $(f,g)$ of $\Z^2$ such that $e=f+g$. Furthermore $f$ and $g$ belong to the same (closed) quadrant of the plane as $e$.
\end{proposition}

\begin{proof}
Let $R$ be the rotation of $\pi /2$. The image $(R f, R g)$ of a direct acute basis of $\Z^2$ still is one. Also, $R$ cyclically permutes the four quadrants of the plane. Without loss of generality, we may thus assume that $a$ and $b$ are positive.

Existence.  
Consider a Bezout relation: $u,v \in \Z^2$ such that $a v - b u = 1$. For any $k \in \mZ$, one also has the relation $a (v+k b) - b (u+k a) = 1$. By euclidean division, and up to such a transformation, we may therefore assume that $0 \leq u < a$. Then $a v = 1+ b u \leq 1+b (a-1) \leq a b$, thus $0 < v \leq b$. 
The vectors $f := (a-u,b-v)$ and $g  := (u,v)$ have non-negative entries. Hence they belong to the same quadrant as $e$, and satisfy $\<f,g\> \geq 0$. Also $\det(f,g) = (a-u) v - (b-v) u = a v-b u = 1$. This concludes the proof of existence.

Uniqueness. Let $(f',g')$ be another direct acute basis such that $e=f'+g'$. We introduce the coordinates $(u',v')$ of $g'$, and observe that $f' = (a-u',b-v')$. Then $\det(f',g') = (a-u')v' - (b-v') u' = a v' -b u'$. We recognize another Bezout relation between the co-prime integers $a,b$. Hence $u' = u+k a$ and $v' = v+k b$ for some $k \in \Z$. Recall that $0 \leq u < a$ and $0 < v \leq b$. If $k<0$, then the coordinates of $g'=(u',v')=g+k e$ satisfy $u \leq u-a '<0$, $v'\leq v-b \leq 0$, while both coordinates of $f' = f-k e$ are positive; this contradicts the assumption $\<f',g'\>\geq 0$. The case $k>0$ is excluded by a similar argument, exchanging the roles of $f'$ and $g'$. Hence $k = 0$ which concludes the proof of uniqueness.
\end{proof}

\subsection{Connected components of the graph $\mT$}
\label{sec:ComponentsT}

We identify the structure of the graph $\mT$, as announced Proposition \ref{prop:V8T}.

\begin{lemma}
\label{lem:Quadrant}
All edges of $\mT$ have both their endpoints in the interior of the same quadrant.
\end{lemma}

\begin{proof}
Any edge of $\mT$ has the form $e \to f\oplus e$ or $e \to e \oplus g$, where $e = f\oplus g$. By proposition \ref{prop:Quadrant}, $f,g$ belong to the same quadrant as $e$. Since both coordinates of $e$ are non-zero, it belongs to the interior of its quadrant. Since this quadrant is a convex cone, the edge joins as announced two point of its interior. 
\end{proof}

\begin{lemma}
\label{lem:SubSum}
Let $e=f \oplus g$. If $\|f\|>\|g\|$ then $f = (f-g)\oplus g$. If $\|g\|>\|f\|$ then $g = f\oplus (g-f)$. If $\|f\|=\|g\|$ then $\|e\|^2 = 2$.
\end{lemma}

\begin{proof}
Since $(f,g)$ is a direct basis, one has $\det(f,g)=1$. Hence $\<f,g\>^2+1 = \|f\|^2\|g\|^2$ by \eqref{eq:ComplexModuleMultiplication}.

If $\|f\|>\|g\|$ then $\<f,g\>^2 +1 > \|g\|^2 \|g\|^2$, thus $\<f,g\> \geq \|g\|^2$, and therefore $\<f-g,g\>\geq 0$. Remarking in addition that $\det(f-g,f) = \det(f,g)=1$, we obtain as announced $f=(f-g)\oplus g$. The case $\|g\|>\|f\|$ is similar.

If $\|f\|=\|g\|$, then $\<f,g\>^2$ and $\|f\|^2\|g\|^2$ are consecutive perfect squares, hence equal to  $0$ and $1$. Thus $\<f,g\>=0$, $\|f\|^2 = \|g\|^2=1$, and therefore $\|e\|^2 = \|f+g\|^2=2$, as announced. 
\end{proof}

\begin{lemma}
\label{lem:AtMostOne}
Let $e = f \oplus g$. If $\|e\|^2=2$, then no edge of $\mT$ arrives at $e$. If $\|e\|^2>2$, then exactly one edge of $\mT$ arrives at $e$, and it must be either $f \to e$ or $g \to e$.
\end{lemma}

\begin{proof}
Edges of $\mT$ have the form $e' \to f' \oplus e'$ (resp. $e' \to e' \oplus g'$) where $e' = f'\oplus g'$. If such an edge arrives at $e$, then by uniqueness of the decomposition, one must have $f' = f$ and $e' = g$, thus $g' = g-f$ (resp.\ $e' = f$ and $g' = g$, thus $f' = f-g$). This corresponds to the two announced cases $f \to e$ or $g \to e$. 

If the first case is realized, then $\<f,g-f\> = \<f',g'\> \geq 0$ (resp.\ second case, $\<f-g,g\> \geq 0$).
Assuming for contradiction that the two cases are realized, we obtain by addition $-\|f-g\|^2 \geq 0$, and therefore $f=g$. This contradicts the assumption that $(f,g)$ is a basis of $\mZ^2$.
\end{proof}

Let us summarize the properties of the graph $\mT$. By Lemma \ref{lem:Quadrant} all edges of $\mT$ have their endpoints within the interior of same quadrant. Also, for any $e \in \mT$:
\begin{itemize}
\item If $\|e\|^2 = 1$, then no edge arrives at $e$, or leaves from $e$.
\item If $\|e\|^2 = 2$, then no edge arrives at $e$, but two edges leave from $e$.
\item If $\|e\|^2 > 2$, then one edge arrives at $e$, and two edges leave from $e$.
\end{itemize}
Furthermore the graph $\mT$ is well founded, in the sense that there is no infinite sequence $e_0 \leftarrow e_1 \leftarrow \cdots$. Indeed the presence of an edge $e \to e'$ between two points implies a strict inequality $\|e\|^2 < \|e'\|^2$ on their squared norms, which are positive integers. These properties together characterize a graph of the form described in Proposition \ref{prop:V8T}.

\subsection{Two lemmas from the preprint \cite{Mirebeau:Vn5Iu9VK}}
\label{sec:TwoLemmasFromMir}

\begin{lemma*}[Lemma 2.3 in \cite{Mirebeau:Vn5Iu9VK}]
\label{lem:SuperbaseDecomp}
Let $(e_0,e_1,e_2)$ be a superbase of $\Z^2$, ordered so that $\|e_0\| \geq \max\{\|e_1\|, \|e_2\|\}$ and $\det(e_1,e_2)=1$. Then $-e_0 = e_1 \oplus e_2$.
\end{lemma*}

\begin{proof}
Observing that $\det(e_0, e_1)=1$, we find that the coordinates of $e_0$ are co-prime. Since $e_0,e_1,e_2$ are pairwise non-collinear, at least one of them is not in the set $V_4 := \{(\pm 1, 0), (0, \pm 1)\}$. Since $e_0$ has the largest norm, $e_0 \notin V_4$. By Proposition \ref{prop:Decomp}, there exists a direct acute basis $(e'_1,e'_2)$ such that $-e_0 = e'_1 \oplus e'_2$. 

Since $\det(e_0,e'_1-e_1)=1-1=0$, there exists $k \in \R$ such that $e'_1 = e_1+k e_0$. Since $e_0$ has co-prime coordinates, $k \in \mZ$. If $k>0$, then we observe that $\<e_1,-e_0\> = \<e'_1-k e_0, -e_0\> = \|e'_1\|^2 + \<e'_1,e'_2\> + k \|e_0\|^2 > \|e_0\|^2$, which implies the contradiction $\|e_1\|>\|e_0\|$. If $k<0$, then observing that $e'_2 = -e_0 - e'_1 = e_2-k e_0$ we reach a similar contradiction $\|e_2\|>\|e_0\|$. Thus $k=0$, and therefore $e'_1 = e_1$, $e'_2 = e_2$, which concludes the proof.
\end{proof}

For any $f,g \in \R^2$, we denote $\ConeS(f,g) := \{\lambda f + \mu g; \, \lambda, \mu > 0\}$ (the interior of $\Cone(f,g)$). 

\begin{lemma*}[Lemma 3.2 in \cite{Mirebeau:Vn5Iu9VK}, here Lemma \ref{lem:ParentsTriangle}]
\label{lem:ConsecutiveParents}
Let $e = f' \oplus g'$ and let $(f,g)$ be a direct acute basis of $\Z^2$ such that $e \in \ConeS(f,g)$. Then $f+g$, $f'$, $g'$, belong to the triangle $T:=[e,f,g]$.
\end{lemma*}

\begin{proof}
Let $\alpha, \beta$ denote the coordinates of $e$ in the basis $(f',g')$, which are positive integers by construction. Observing that $1e + (\beta-1) f + (\alpha-1) g = (\alpha+\beta-1) (f+g)$ we obtain as announced that $f+g \in T$.

We fix $e$ and prove that $f',g' \in T:= [e,f,g]$, for any direct acute basis $(f,g)$ such that $e \in \ConeS(f,g)$, by \emph{decreasing} induction on the integer $k := \<f,g\>$.
Initialization. Assuming that $k \geq \frac 1 2 \|e\|^2$, we obtain the impossibility $\|e\|^2 = \|\alpha f + \beta g\|^2 > 2 \alpha \beta \<f,g\> \geq 2 \<f,g\> \geq \|e\|^2$. This case is vacuous, hence true.

Induction. If $e=f+g$, then $e = f\oplus g$ and therefore $f=f'$, $g=g'$; the result follows. Otherwise, we have either $e \in \ConeS(f, f+g)$ or $e \in \ConeS(f+g, g)$. By induction, since $\<f,f+g\> > \<f,g\>$ and $\<f+g,g\> > \<f,g\>$, we obtain that $f',g'$ belong to $T_1 := [e,f,f+g]$ or $T_2 := [e,g,f+g]$. Recalling that $f+g \in T$ we obtain $T_1 \cup T_2 \subset T$ which concludes the proof.
\end{proof}

\bibliographystyle{alpha}
\bibliography{Includes/AllPapers}

\end{document}

%% file: Includes/CalligraphicLetters.tex
\def\cD{{\cal D}}
\def\cE{{\cal E}}
\def\cF{{\cal F}}

\def\cU{{\cal U}}
\def\cV{{\cal V}}

\def\mD{{\mathbb D}}

\def\mQ{{\mathbb Q}}
\def\mR{{\mathbb R}}

\def\mT{{\mathbb T}}
\def\mU{{\mathbb U}}

\def\mZ{{\mathbb Z}}

\def\R{\mR}

\def\Z{\mZ}

%% file: Includes/Operators.tex
\def\ve{\varepsilon}
\def\vp{\varphi}

\newcommand\trans{\mathrm T}
\DeclareMathOperator\Id{Id}

\DeclareMathOperator\length{length}
\DeclareMathOperator\Area{Area}

\def\<{\langle}
\def\>{\rangle}

\def\sm{\setminus}

%% file: Includes/TheoremsEng.tex
\newtheorem{theorem}{Theorem}[section]
\newtheorem{remark}[theorem]{Remark}
\newtheorem{lemma}[theorem]{Lemma}
\newtheorem{proposition}[theorem]{Proposition}
\newtheorem{corollary}[theorem]{Corollary}
\newtheorem{definition}[theorem]{Definition}

\newtheorem*{proposition*}{Proposition}
\newtheorem*{theorem*}{Theorem}
\newtheorem*{lemma*}{Lemma}

%% file: OTBasis_ConsistencyMinimality.tex
\section{Proof of consistency and minimality}
\label{sec:Local}

We establish the results announced in \S \ref{sec:LocalIntro}, and related to the MA-LBR consistency and optimal locality. Theorem \ref{th:Consistency} (Consistency) and Remark \ref{rem:Geom} are proved in \S \ref{sec:Minimality}, and Theorem \ref{th:Minimality} (Minimality) in \S \ref{sec:Minimality}.

\subsection{Consistency}
\label{sec:Consistency}

Our first result, Proposition \ref{prop:hSuperbase} preceded with a technical lemma, shows that the MA-LBR operator \eqref{eq:SB} systematically overestimates the hessian determinant of quadratic functions. For any $M\in S_2^+$, defining $u_M$ as in \eqref{eqdef:UM}, one has $\cD^\SB_V u_M \geq \det(M)$ on $X$. Equality holds iff $V$ contains an $M$-obtuse superbase, which establishes the announced Theorem \ref{th:Consistency} (Consistency). We denote
\begin{align}
\label{eqdef:K}
K &:= \{(a,b,c) \in \R_+^3; \, a \leq b+c, \, b\leq c+a, \, c\leq a+b\}.\\
\label{eqdef:h1}
h_1(a,b,c) &:= \frac 1 2 (ab+bc+ca)-\frac 1 4 (a^2+b^2+c^2).
\end{align}

\begin{lemma}
\label{lem:hh1}
Let $(a,b,c)\in \R_+^3$. Then $h(a,b,c) \geq h_1(a,b,c)$, with equality iff $(a,b,c) \in K$.
\end{lemma}

\begin{proof}
If $(a,b,c) \in K$ then $h(a,b,c)=h_1(a,b,c)$ by definition \eqref{eqdef:h}. Otherwise, we may assume without loss of generality that $a > b+c$, so that $h(a,b,c) = b c$ and $h(a,b,c) - h_1(a,b,c) = \frac 1 4 (a-b-c)^2 > 0$.
\end{proof}

\begin{proposition}
\label{prop:hSuperbase}
Let $M \in S_2^+$, let $(e_0, e_1, e_2)$ be a superbase of $\Z^2$, and let $\delta_i := \<e_i, M e_i\>$ for $0 \leq i \leq 2$. Then $h(\delta_0, \delta_1, \delta_2) \geq \det(M)$. Equality holds iff $(\delta_0, \delta_1, \delta_2) \in K$, equivalently iff $(e_0, e_1, e_2)$ is $M$-obtuse.
\end{proposition}

\begin{proof}
Given a permutation $\{i,j,k\}$ of $\{0,1,2\}$ we compute 
\begin{equation*}
\delta_i - \delta_j -\delta_k = \<e_j+e_k, M (e_j+e_k)\> - \<e_j, M e_j\> - \<e_k, M e_k\> = 2\<e_j, M e_k\>.
\end{equation*}
Hence $(\delta_0, \delta_1, \delta_2) \in K$ iff the superbase $(e_0, e_1, e_2)$ is $M$-obtuse. We prove in the following that $h_1(\delta_0, \delta_1, \delta_2) = \det(M)$, which in view of Lemma \ref{lem:hh1} concludes the proof.

Special case of the superbase $f_0 := (-1,-1)$, $f_1:=(1,0)$, $f_2:=(0,1)$, with $\mu_i := \<f_i, M f_i\>$. We get $\mu_0 = M_{11}+2M_{12}+M_{22}$, $\mu_1=M_{11}$, $\mu_2 = M_{22}$. Inserting this into the expression \eqref{eqdef:h1}, yields as announced $h_1(\mu_0, \mu_1, \mu_2) = M_{11} M_{22} - M_{12}^2 = \det(M)$.

General case. Let $A$ be a matrix such that $A f_1= e_1$ and $A f_2 = e_2$, so that by linearity $A f_0 = e_0$. Note that $|\det(A)|=|\det(e_1,e_2)/\det(f_1,f_2)|=1$. We obtain $\delta_i = \<f_i, A^\trans M A f_i\>$, for all $0\leq i \leq 2$, so that by the special case $h_1(\delta_0,\delta_1,\delta_2) = \det(A^\trans M A) = \det(M)$.
\end{proof}

The rest of this subsection is devoted to the proof of Remark \ref{rem:Geom}. 
Let $(e_0,e_1,e_2)$ be a fixed superbase of $\mZ^2$. For each $\delta = (\delta_0,\delta_1,\delta_2) \in \mR_+^3$ we introduce a polygon $H(\delta)$, defined by linear inequalities, and some of its edges $E_i(\delta)$, $1 \leq i \leq 3$
\begin{align*}
H(\delta) &:= \{l \in \mR^2; \forall \, 1 \leq i \leq 3, \ |\<l, e_i\>| \leq \delta_i\}.\\
E_i(\delta) &:= \{l \in H(\delta); \, \<l, e_i\> = \delta_i\}.
\end{align*}
The area of $H(\delta)$ is computed in Corollary \ref{corol:SubGradientArea}, and this polygon (properly scaled and translated) is identified with a subgradient set in Proposition \ref{prop:SubgradientSet}, concluding the proof of Remark \ref{rem:Geom}. The proof unfortunately gives little geometric insight, hence it could be skipped at first reading.
Given $A \subset \mR^n$, $x \in \mR^n$ and $\alpha \in \mR$ we use the notation $x+\alpha A := \{x+\alpha a; a\in A\}$.
\begin{lemma}
\label{lem:EdgeLengths}
Let $\delta \in \mR_+^3$. Then $E_0(\delta)$ is a segment of length (i) $(\delta_1+\delta_2-\delta_0) \|e_0\|$ if $\delta \in K$, (ii) $2  \delta_2 \|e_0\|$ if $\delta_1 \geq \delta_0+\delta_2$, (iii) $2 \delta_1 \|e_0\|$ if $\delta_2 \geq \delta_0+\delta_1$, or (iv) $0$ if $\delta_0 > \delta_1+\delta_2$ (a case where $E_0(\delta)$ is in fact empty).
\end{lemma}

\begin{proof}
Let $x,y \in \mR^2$ and let $l := x e_0+y e_0^\perp$. One has $l \in E_0(\delta)$ iff 
\begin{equation*}
\<x e_0 + y e_0^\perp, e_0\> = \delta_0, \text{ and } |\<x e_0 + y e_0^\perp, e_i\>| \leq \delta_i, \ i \in \{1,2\}.
\end{equation*}
The equality is equivalent to $x = \delta_0/\|e_0\|^2$. Recall that $\<e_0^\perp,e_1\>= \det(e_0,e_1)=1$, and likewise $\<e_0^\perp,e_2\>= \det(e_0,e_2)=-1$. Hence the  inequalities respectively hold iff $y$ belongs to the segment 
\begin{equation*}
S_1 := -\frac{\<e_0,e_1\>\delta_0}{\|e_0\|^2} + [-\delta_1,\delta_1], \qquad S_2 := \frac{\<e_0,e_2\>\delta_0}{\|e_0\|^2} + [-\delta_2,\delta_2].
\end{equation*}
Translating these two segments by $\<e_0,e_1\>\delta_0/\|e_0\|^2$ yields $F_1 = [-\delta_1,\delta_1]$, and $F_2 = [-\delta_2-\delta_0, \delta_2-\delta_0]$. Finally the length of $E_\delta$ is 
\begin{equation}
\|e_0\| \times \length(F_1 \cap F_2) = \|e_0\| \left(\min\{\delta_1,\delta_2-\delta_0\} + \min\{\delta_1, \delta_2+\delta_0\}\right)_+,
\end{equation}
which coincides with the announced result.
\end{proof}


\begin{corollary}
\label{corol:SubGradientArea}
For any $\delta = (\delta_0, \delta_1, \delta_2)\in \mR^+_3$, one has 
$\Area(H(\delta)) = 4h(\delta)$. 
\end{corollary}

\begin{proof}
The triangle $\Hull(E_i(\delta) \cup \{0\})$ has area $\frac 1 2 \times \frac {\delta_i} {\|e_i\|} \times \length(E_i(\delta))$: half the height from the vertex at the origin, times the length of the opposite side. These three (possibly empty) triangles, with their opposites, partition $H(\delta)$. From this point the result follows from Lemma \ref{lem:EdgeLengths} and an easy calculation.
\end{proof}

\begin{proposition}
\label{prop:SubgradientSet}
Let $M \in S_2^+$, and let $\delta_i := \<e_i, M e_i\>$, $1 \leq i \leq 3$. Let $x \in \mZ^2$, and let $U$ be the maximal convex map bounded above by $u_M$ at the points $x$ and $x \pm e_i$, $1 \leq i \leq 3$. Then $\partial U(x) = M x + \frac 1 2 H(\delta)$.
\end{proposition}

\begin{proof}
For any $g \in \mR^2$, the following are equivalent:
\begin{itemize}
\item $g \in \partial U(x)$ 
\item $U(x) + \<g, p-x\> \leq U(p)$, for all points $p$ of the hexagon of vertices $x \pm e_i$, $1 \leq i \leq 3$.
\item $u_M(x) + \<g,p-x\> \leq u_M(p)$, for all $p = x+ \ve e_i$, $\ve \in \{-1,1\}$, $1 \leq i \leq 3$. 
\end{itemize}
In order to further simplify this expression, we write $g = M x+l/2$, $p = x+e$, where $e = \pm e_i$, $1 \leq i \leq 3$, and insert the expression \eqref{eqdef:UM} of $u_M$. The following are then equivalent:
\begin{align*}
u_M(x) + \<g,p-x\> &\leq u_M(p), \\
\<x, M x\> + 2 \<Mx + l/2, (x+e)-x\> &\leq \<x+e, M (x+e)\>,\\
\<l, e\> &\leq \<e, M e\>.
\end{align*}
We recognize the inequalities defining $H(\delta)$, and the announced result follows.
\end{proof}

\begin{proof}[Proof of remark \ref{rem:Geom}]
Proposition \ref{prop:SubgradientSet} and Corollary \ref{corol:SubGradientArea} imply as announced that $\Area( \partial U(x) ) = \frac 1 4 \Area(H(\delta)) = h(\delta)$. 
By Proposition \ref{prop:hSuperbase} one has $h(\delta)=\det M$ iff $\delta \in K$, which by Lemma \ref{lem:EdgeLengths} means that $H(\delta)$ is an hexagon: each edge $E_i(\delta)$ has a positive length (we exclude here for simplicity the limit case $\delta \in \partial K$).

The map $U$ is polygonal on a triangulation with vertices $x$, $x\pm e_i$, $1 \leq i \leq 3$, which is symmetric with respect to $x$. Only four such triangulations exist, as illustrated on Figure \ref{fig:Subgradient}, and only the first one leads to an hexagonal subgradient $\partial U(x)$, since the subgradient has one vertex for each triangle containing $x$. This concludes the proof.
\end{proof}

%
%
%

\subsection{Minimality}
\label{sec:Minimality}

We prove Theorem \ref{th:Minimality} (Minimality), on the optimal locality of the MA-LBR. For that purpose we introduce some definitions, and establish in Proposition \ref{prop:ObtuseNorm} a minimality property of obtuse superbases.

\begin{definition}
\label{def:TrigConsecutive}
We denote by $\Cone(f,g) := \{\alpha f + \beta g;\, \alpha, \beta \in \R_+\}$, the closed convex cone spanned by two elements $f,g \in \R^2$. 
We say that $f,g$ are \emph{trigonometrically consecutive} elements of a set $V \subset \R^2 \sm\{0\}$ iff they are not collinear and no element of $V$ lies in the interior of $\Cone(f,g)$.
\label{def:Cone}
\end{definition}

\begin{definition}
A matrix $M \in S_2^+$ is said \emph{generic} iff there exists no $M$-orthogonal basis of $\Z^2$. (i.e. $(f,g) \in (\Z^2)^2$ such that $|\det(f,g)|=1$ and $\<f, M g\>=0$.)
\end{definition}

\begin{proposition}
\label{prop:ObtuseNorm}
Let $M\in S_2^+$, and let $(e_0,e_1,e_2)$ be an $M$-obtuse superbase. 
Then for each $e \in \Z^2 \sm \{\pm e_0, \pm e_1, \pm e_2\}$ with $\gcd(e) = 1$ one has $\|e\|_M \geq \max\{\|e_0\|_M, \|e_1\|_M, \|e_2\|_M\}$. The inequality is strict if $M$ is generic.
\end{proposition}

\begin{proof}
Consider the set $S := \{e_0,-e_2, e_1, -e_0, e_2,-e_1\}$, where for convenience elements are ordered trigonometrically, and some $e \in \Z^2\sm S$ with $\gcd(e)=1$. Let $f,g$ be trigonometrically consecutive elements of $S$ such that $e \in \Cone(f,g)$. Since $(f,g)$ is a basis of $\Z^2$, there exists $\alpha, \beta \in \mZ$ such that $e = \alpha f+ \beta g$. Since $e \in \Cone(f,g)$, we have $\alpha, \beta \geq 0$. Since $\gcd(e)=1$, one has $\gcd(a,b)=1$. Since $e \notin S$, one has $\alpha, \beta \geq 1$. By construction of $S$ one has $\<f,M g\> \geq 0$, hence $\|e\|_M^2 \geq \|f\|_M^2+\|g\|_M^2 \geq \max\{\|f\|_M^2, \|g\|_M^2, \|f-g\|_M^2\} =  \max\{\|e_0\|_M, \|e_1\|_M, \|e_2\|_M\}^2$ as announced. If $M$ is generic then $\<f, M g\> \neq 0$, thus $\<f, M g\> > 0$, hence inequalities are strict. 
\end{proof}

We next study trigonometrically consecutive elements $f,g$ of the stencil $V$ of an operator $\cD$ which consistency set contains a given matrix $M$. 
Corollary \ref{corol:ConsecutiveScal}, preceded with a technical lemma, identifies the sign of the scalar product $\<f,M g\>$. 
\begin{lemma}
\label{lem:Angle}
Let $M\in S_2^+$, and let $f,g\in \R^2$ be non collinear and such that $\<f,M g\> <0$. Then there exists $N \in S_2$ such that (i) $\det(M+\delta N) > \det(M)$ for any sufficiently small $\delta>0$, and (ii) $\<e, N e\> \leq 0$ for all $e \in \Cone(f,-g)$. 
\end{lemma}

\begin{proof}
Case $M = \Id$ and $\|f\|=\|g\|=1$. We define $N \in S_2$ by the (non-definite) quadratic form $\<x,N x\> = \det(f,x) \det(x,g)$, $x \in X$.
It has eigenvectors $f+g$ and $f-g$, by a symmetry argument, with respective eigenvalues $ \lambda_0 := \det(f,g)^2/\|f+g\|^2$, and $-\lambda_1$ where $\lambda_1:=\det(f,g)^2/\|f-g\|^2$. Since $\<f,g\> < 0$ we have $\|f+g\| < \|f-g\|$, hence $\lambda_0 > \lambda_1$ and therefore $\det(\Id+ \delta N) = 1+\delta (\lambda_0-\lambda_1)  - \delta^2 \lambda_0\lambda_1 > 1$ for small positive $\delta$ as announced. 

General case. Write $M = A^\trans A$, for some $2 \times 2$ invertible matrix $A$, and take $N = A^\trans N' A$ where $N'$ is associated to $\Id$, $A f /\|A f\|$, $A g/\|A g\|$.
\end{proof}

\begin{corollary}
\label{corol:ConsecutiveScal}
Let $\cD$ be a DE2 operator with stencil $V$, and which consistency set contains the neighborhood of a matrix $M\in S_2^+$. 
If $f,g$ are trigonometrically consecutive elements of $V$, then $\<f,M g\> \geq 0$.
\end{corollary}

\begin{proof}
Assume for contradiction that $\<f,M g\> < 0$. Let $N \in S_2$ be given by Lemma \ref{lem:Angle}, and let $M_\delta := M + \delta N$ for some small $\delta \geq 0$. An element $e \in V$ cannot belong to the interior of $\Cone(f,g)$ by definition, and neither to the interior of $\Cone(-f,-g)$ by symmetry of $V$. Hence it belongs to $\Cone(f,-g)$ or $\Cone(-f,g)$, which implies  $\<e, N e \> \leq 0$. We have obtained that $\<e, M_\delta e\> \leq \<e, M e\>$ for all $e \in V$, so that  by degenerate ellipticity $\det(M_\delta) = \cD(u_{M_\delta}) \leq \cD(u_M) = \det(M)$. 
This contradicts Lemma \ref{lem:Angle}, which concludes the proof.
\end{proof}

Our following step, Corollary \ref{corol:ConsecutiveDet} preceded with a technical lemma, shows that without loss of generality one can assume that consecutive elements of a stencil $V$ form bases of $\Z^2$.

\begin{lemma}
\label{lem:TriDet}
Let $f,g \in \mZ^2$, and let $T$ be the triangle of vertices $0,f,g$. If $|\det(f,g)|>1$ then $T$ contains a point $e$ distinct from its vertices, and such that $\gcd(e)=1$.
\end{lemma}

\begin{proof}
Since $|\det(f,g)|>1$ the map $(\alpha, \beta) \in \Z^2 \mapsto \alpha f + \beta g \in \Z^2$ is not surjective. Hence there exists $(\alpha, \beta) \in \mQ^2$, at least one of them non-integer, such that $\alpha f + \beta g \in \Z^2$. Up to replacing $(\alpha, \beta)$ with $(\alpha - m, \beta - n)$, $(m,n) \in \Z^2$, we may assume that $\alpha, \beta \in [0,1]$. Up to replacing $(\alpha, \beta)$ with $(1-\alpha, 1-\beta)$, we may assume that $\alpha+ \beta \leq 1$.
The point $e := \alpha f + \beta g \in \Z^2$ belongs to $T$ and is distinct from its vertices. In the case where $\gcd(e) > 1$, we can replace it with $e/\gcd(e)$.
\end{proof}

\begin{corollary}
\label{corol:ConsecutiveDet}
Let $\cD$ be a DE2 operator with stencil $V$, and with a consistency set of non-empty interior. Then there exists a DE2 operator $\cD'$ with stencil $V'$, such that (i) $\cD$ and $\cD'$ have the same consistency set, (ii) $\Hull(V') \subset \Hull(V)$, and (iii) any two trigonometrically consecutive elements $f,g \in V'$ satisfy $|\det(f,g)| = 1$.
\end{corollary}

\begin{proof}
Let $V' := \{ e \in \Hull(V) \cap \Z^2; \, \gcd(e) = 1\}$. For each $e \in V$ one has $e' := e/\gcd(e) \in \Hull(\{-e,e\})$, hence $e' \in V'$ since $V$ is symmetric w.r.t the origin. Note that $\Delta_e u_M = \<e,M e\> = \gcd(e)^2 \Delta_{e'} u_M$, for any $M \in S_2^+$, using \eqref{eqdef:UM}. Constructing $\cD'$ in terms of $\cD$ is from this point straightforward.

Let $f,g \in V'$ be trigonometrically consecutive. If $|\det(f,g)|>1$ then Lemma \ref{lem:TriDet} provides a point $e \in [0,f,g] \subset \Hull(V)$ with $\gcd(e) =1$, hence $e \in V'$; this contradicts our assumption on $f,g\in V'$.
The case $\det(f,g)= 0$ is excluded by Definition \ref{def:TrigConsecutive}, hence $|\det(f,g)|=1$ which concludes the proof.
\end{proof}

Finally, we identify a condition under which a stencil $V$ contains an $M$-obtuse superbase, and we conclude the proof of the announced Theorem \ref{th:Minimality}.

\begin{lemma}
\label{lem:AcuteStencilSuperbase}
Let $M\in S_2^+$, and let $V$ be a stencil which contains some non-collinear elements, and such that any two trigonometrically consecutive $f,g\in V$ satisfy $\<f, M g \> > 0$ and $|\det(f,g)|=1$. Then $V$ contains an $M$-obtuse superbase.
\end{lemma}

\begin{proof}
Let $e$ be an element of an $M$-obtuse superbase, and let $f,g \in V$ be trigonometrically consecutive and such that $e \in \Cone(f,g)$. Note that $f,g$ exist because $V$ contains some non-collinear elements and is symmetric w.r.t the origin.
Since $|\det(f,g)|=1$, one has $e = \alpha f+ \beta g$ for some $\alpha, \beta \in \Z$. Since $e \in \Cone(f,g)$, we have $\alpha, \beta \geq 0$.
Since $\gcd(e)=1$, one has $\gcd(\alpha, \beta)=1$.
Assuming for contradiction that $\alpha, \beta \geq 1$, we obtain $\|e\|_M^2 > \|f\|^2_M+\|g\|^2_M > \max \{\|f\|_M^2, \|g\|_M^2, \|f-g\|_M^2\}$ since $\<f, M g\> >0$. This contradicts Proposition \ref{prop:ObtuseNorm}, therefore $\alpha\beta = 0$. But then $(\alpha, \beta)$ equals $(1,0)$ or $(0,1)$, since $\gcd(\alpha, \beta)=1$. Thus $e \in \{f,g\} \subset V$ which concludes the proof.
\end{proof}

\begin{proof}[Proof of Theorem \ref{th:Minimality}]
Let $\cD$ be a DE2 operator with stencil $V$, and which consistency set contains the neighborhood of a generic matrix $M$. Let $\cD'$ and $V'$ be as described in Corollary \ref{corol:ConsecutiveDet}. Let $f,g \in V'$ be trigonometrically consecutive; note that $|\det(f,g)|=1$. 

Case of a generic matrix $M \in S_2^+$.
Corollary \ref{corol:ConsecutiveScal} states that $\<f, M g \> \geq 0$, hence $\<f, M g \> > 0$ since $M$ is generic. The consistency assumption implies that $V$ contains non-collinear elements, hence so does $V'$. Invoking Lemma \ref{lem:AcuteStencilSuperbase} we find that $V' \subset \Hull(V)$ contains an $M$-obtuse superbase, as announced.

Case of a non-generic $M \in S_2^+$. Let $(M_n)_{n \geq 0}$, $M_n \in S_2^+$, be a sequence of generic matrices converging to $M$. By the previous point, $\Hull(V)$ contains an $M_n$-obtuse superbase $(e^n_0, e^n_1, e^n_2)$ for all sufficiently large $n$. By Proposition \ref{prop:convalg2} the elements of these superbases are bounded independently of $n$. Since superbases are discrete objects, infinitely many among this sequence are equal to some fixed $(e_0,e_1,e_2) \in (\mZ^2)^3$, also contained in $\Hull(V)$ and which by continuity is an $M$-obtuse superbase.
\end{proof}

%% file: OTBasis_HierarchicalStencilRefinement.tex
\section{Proofs on hierarchical stencil refinement}
\label{sec:Hierarchy}

We establish the results announced in \S \ref{sec:HierarchyIntro}. Propositions \ref{prop:VExtended} (Stencil extension) and \ref{prop:PositiveDifferences} (Discrete convexity) are proved in \S \ref{sec:VExtended}. Algorithm \ref{algo:D2p} is rephrased in \S \ref{sec:Algorithm} as a depth first search within the Stern-Brocot tree. Theorem \ref{th:Hierarchy} (Adaptive pruning equals extensive sweeping) is established in \S \ref{sec:PropagationOnTrees} in the quadratic case, and in \S \ref{sec:HierarchyConclusion} in the general case.

\subsection{Properties of stencils}
\label{sec:VExtended}

We establish several properties of stencils announced in \S \ref{sec:HierarchyIntro}, starting with Proposition \ref{prop:VExtended}: any stencils $\cV$ can be extended by union with the sets $\cV_\Omega$. This requires two technical lemmas.
\begin{lemma}
\label{lem:DecompEdge}
Let $e \in \mT \sm V_8$, $e = f \oplus g$. The graph $\mT$ has exactly one edge arriving at $e$, which is either $f \to e$ or $g \to e$.
\end{lemma}

\begin{proof}
The existence of a unique edge arriving at $e$ follows from the description of $\mT$, Proposition \ref{prop:V8T}.
Let $e' = f'\oplus g' \in \mT$ be such that $e' \to e$. By definition of this graph structure $e$ equals $f'\oplus e'$ or $e'\oplus g'$. By uniqueness of the decomposition $e' \in \{f,g\}$, which concludes the proof.
\end{proof}

\begin{lemma}
\label{lem:AllSensibleDirections}
Let $\cV$ be a family of stencils, and let $x\in X$. Then any $e \in \mT$ such that $x\pm e \in \Omega$, belongs to $\cVu(x) := \cV(x) \cup \cVin(x)$.
\end{lemma}

\begin{proof}
If $e \notin \cVin(x)$, then $e \in \cV(x)$ by (Reachability).
\end{proof}

\begin{proof}[Proof of Proposition \ref{prop:VExtended}]
We consider a family $\cV$ of stencils, and show that $\cVu$ also is one. The inclusion $\cV(x) \subset \cVu(x)$ implies (Reachability), as well as $V_8 \subset \cVu(x)$. The inclusions $\cV(x) \subset \mT$ and $\cVin(x) \subset \mT$ imply that $\cVu(x) \subset \mT$.  Only (Hierarchy) is thus left to prove.

Consider an edge $e' \mapsto e$ of $\mT$, with $e \in \cVu (x)$. Our objective is to show that $e' \in \cVu(x)$. If $e \in \cV(x)$, then this follows from (Hierarchy) for $\cV$. Otherwise $e \in \cVin(x)$, hence it admits a decomposition $e = f \oplus g$, and $e' \in \{f,g\}$ by Lemma \ref{lem:DecompEdge}. Then $x\pm e' \in \Omega$ by definition \eqref{eqdef:Vin} of $\cVin(x)$, and therefore $e' \in \cVu(x)$ by Lemma \ref{lem:AllSensibleDirections}.
\end{proof}

Our next proposition shows as announced in \S \ref{sec:HierarchyIntro} that condition (Reachability), required for families of stencils, is vacuous for all points $x \in X$ far from the boundary $\partial \Omega$. Corollary \ref{corol:ReachabilityVacuousBox} shows in addition that it is entirely vacuous if $\Omega$ is a box domain. For that purpose we need a technical lemma.

\begin{lemma}
\label{lem:ParentsTriangle}
Let $e = f \oplus g$, and let us assume that $e$ has positive coordinates. Then $f,g$ belong to the triangle $[(1,0),(0,1),e]$.
\end{lemma}

\begin{proof}
The proof appears in \cite{Mirebeau:Vn5Iu9VK}, but is reproduced Appendix \ref{sec:TwoLemmasFromMir} for completeness.
\end{proof}

\begin{proposition}
\label{prop:ReachabilityVacuous}
Let $x \in X$ be such that $x\pm (1,0), x\pm (0,1) \in \Omega$. Then condition (Reachability) is vacuous for $x$, in the following sense: any $e\in \mT \sm \cVin(x)$ such that $x \pm e \in \Omega$ must be of the form $(\pm 1, 0)$ or $(0, \pm 1)$, hence automatically $e \in V_8 \subset \cV(x)$.
\end{proposition}

\begin{proof}
Consider $x \in X$ and $e \in \mT$, distinct from $(\pm 1,0)$ and $(0, \pm 1)$ and such that $x \pm e \in \Omega$.
Without loss of generality, we assume that both coordinates of $e$ positive, and that $x=0$. By Proposition \ref{prop:Decomp} we may introduce the decomposition $e = f \oplus g$. By Lemma \ref{lem:ParentsTriangle}, $f,g \in [(1,0),(0,1),e] \subset \Omega$. Likewise $-f,-g \in [(-1,0), (0,-1),-e] \subset \Omega$. Thus $e \in \cVin(x)$, which concludes the proof.
\end{proof}

\begin{corollary}
\label{corol:ReachabilityVacuousBox}
Assume a box domain $\Omega = ]a_1^-,a_1^+[\times ]a_2^-,a_2^+[$. Then condition (Reachability) is vacuous for all $x\in X$, in same sense as in Proposition \ref{prop:ReachabilityVacuous}.
\end{corollary}

\begin{proof}
Let $x = (x_1,x_2) \in X$, and let $e =(e_1,e_2) \in \mT \sm \cVin(x)$, distinct from $(\pm 1,0)$ and $(0, \pm 1)$ and such that $x\pm e \in \Omega$. Since $\gcd(e) = 1$, both $e_1$ and $e_2$ are non-zero integers. Since $x\pm e \in \Omega$, we have $a_i^- < x_i - |e_i|$, $a_i^+ > x_i + |e_i|$, for each $i \in \{1,2\}$. Hence $x\pm (1,0), x\pm (0,1) \in \Omega$. Applying Proposition \ref{prop:ReachabilityVacuous} we conclude the proof.
\end{proof}

We conclude this section with the proof of Proposition \ref{prop:PositiveDifferences}, in Corollaries \ref{corol:PositiveDifferences1} and \ref{corol:PositiveDifferences2}, which ties the positivity of the MA-LBR operator with a local discretization of convexity.

\begin{lemma}
\label{lem:hDiff}
The map $h$ is non-decreasing in all its variables on $\R_+^3$. For all $a,b,c \in \R_+$ we have $\min \{ab, bc, ca\} \geq h(a,b,c) \geq \frac 3 4 \min\{a,b,c\}^2$. 
\end{lemma}

\begin{proof}
First point. One easily checks that the piecewise definitions \eqref{eqdef:h} of $h(a,b,c)$ agree on the interface $\partial K$ \eqref{eqdef:K}, i.e.\ when $a=b+c$ or $b=a+c$ or $c=a+b$; hence $h$ is continuous. We then compute $\nabla h(a,b,c) = (b+c-a,\, a+c-b,\, a+b-c)/2$ for all $(a,b,c)\in K$, and $\nabla h(a,b,c) = (0,c,b)$ when $a\geq b+c$ (resp.\ likewise permuting the roles of $a,b,c$). Hence the components of $\nabla h$ are non-negative everywhere, and therefore $h$ is non-decreasing in all its variables.

Second point: If $a \geq b+c$ then $h(a,b,c) = b c$ by definition \eqref{eqdef:h}, and otherwise $h(a,b,c) \leq h(b+c,b,c) = b c$ by the first point. Hence $h(a,b,c) \leq b c$ for all $a,b,c \in \mR^+$. Likewise $h(a,b,c) \leq bc$ and $h(a,b,c) \leq ca$, thus $h(a,b,c) \leq \min\{ab,bc,ca\}$ as announced. Finally, denoting $\delta := \min \{a,b,c\}>0$ we obtain $h(a,b,c) \geq h(\delta, \delta, \delta) = \frac 3 4 \delta^2$.
\end{proof}

\begin{corollary}
\label{corol:PositiveDifferences1}
Let $u \in \FX$, let $\cV$ be a family of stencils, and let $x \in X$. If $\Delta_e u(x)>0$ for all $e \in \cV(x)$ then $\cD_\cV u (x) > 0$.
\end{corollary}

\begin{proof}
The operator value $\cD_\cV u(x)$ is the minimum \eqref{eqdef:DVux} of a finite collection of terms of the form $\huefg$, where $e,f,g \in \cV(x)$, and which by Lemma \ref{lem:hDiff} are positive.
\end{proof}

\begin{lemma}
\label{lem:ReversedHierarchy}
Let $\cV$ be a family of stencils, and let $x \in X$. 
If $e \in \cV(x)$ and $e=f\oplus g$, then $f,g \in \cV(x)$.
\end{lemma}

\begin{proof}
Denote $V := \cV(x)$. We proceed by induction on the integer $\|e\|^2$. If $\|e\|^2 = 1$, then it admits no decomposition of the form $f\oplus g$. If $\|e\|^2 = 2$, then $e \in V_8$, and therefore $f,g \in V_8 \subset V$.

If $\|e\|^2>2$ then by Proposition \ref{prop:V8T} the graph $\mT$ has an edge $e' \to e$.
We write $e' = f' \oplus g'$. By (Hierarchy) one has $e' \in V$, and by induction $f',g' \in V$. By definition of $\mT$ the vector $e$ is either $f'\oplus e'$, or $e'\oplus g'$. Hence $\{f,g\} \subset \{e',f',g'\} \subset V$ which concludes the proof.
\end{proof}

\begin{corollary}
\label{corol:PositiveDifferences2}
Let $u \in \FX$, let $\cV$ be a family of stencils, and let $x \in X$.
Then 
\begin{equation}
\label{eq:DDecomp}
\cD_\cV u(x) = 
\min\{ \huefg;\, e \in \cV(x),\, e=f \oplus g \}. 
\end{equation}
Also, if $\cD_\cV u (x) > 0$ then $\Delta_e u(x)>0$ for all $e \in \cV(x)$.
\end{corollary}

\begin{proof}
First point. Note the symmetries (i) $\Delta_e u(x) = \Delta_{-e} u(x)$ for any $e \in \mZ^2$, and (ii) $e \in \cV(x)$ iff $-e \in \cV(x)$ by Definition \ref{def:Stencil}. We denote by $D$ (resp.\ $D'$) the left (resp.\ right) hand side of \eqref{eq:DDecomp}. 
If $e \in \cV(x)$ and $e=f\oplus g$, then $(e,-f,-g)$ is a superbase of $\mZ^2$ and $f,g \in \cV(x)$ by Lemma \ref{lem:ReversedHierarchy}; hence $D \leq D'$. Conversely let $(e,f,g) \in \cV(x)^3$ be a superbase of $\mZ^2$. Up to reordering these vectors we may assume that $\|e\| \geq \max \{\|f\|, \|g\|\}$ and $\det(f,g)=1$. Then $-e = f \oplus g$ by Lemma \ref{lem:SuperbaseDecomp}, which implies $D' \leq D$ and establishes \eqref{eq:DDecomp}.

Second point. Let $e \in \cV(x)$. If $e=f \oplus g$ then $0 < \cD_\cV u(x) \leq \huefg$ by \eqref{eq:DDecomp}, and therefore $\Delta_e u(x) >0$ by Lemma \ref{lem:hDiff}. Otherwise, $e$ is among $(\pm 1,0)$ or $(0, \pm 1)$, hence we may choose $f,g \in V_8 \subset \cV(x)$ such that $e,f,g$ is a superbase. Using \eqref{eqdef:DVux} and Lemma \ref{lem:hDiff} we again obtain $\Delta_e u(x) >0$.
\end{proof}

\subsection{Depth-first exploration within the Stern-Brocot tree}
\label{sec:Algorithm}

In this section, we interpret the MA-LBR operator $\cDu_\cV$ defined in Algorithm \ref{algo:D2p} as a depth-first transversal of a subtree of the Stern-Brocot tree. 
The concept of depth-first exploration is introduced in Algorithm \ref{algo:DepthFirst}.

\begin{algorithm}
\caption{Depth-first exploration of a finite ordered tree $T$, with root $e_*$}
\label{algo:DepthFirst}
\begin{tabular}{l}
\textbf{Initialize} a mutable list $L \leftarrow [e_*]$.\\
\textbf{While} L is non empty \textbf{do}\\
\phantom{bla} Remove from $L$ its first element $e$, and denote by $e_1, \cdots, e_n$ its children in the tree $T$.\\
\phantom{bla} Prepend $[e_1,\cdots,e_n]$ to $L$.
\end{tabular}
\end{algorithm}

We introduce in Algorithm \ref{algo:MinTree} a simplified version of the adaptive MA-LBR operator $\cDu_\cV$. It incorporates a dummy variable $L$ used to emulate Algorithm \ref{algo:DepthFirst}, see Proposition \ref{prop:MinTreePlus} below. $G_i$ denotes the $i$th element of the mutable list $G$. 

\begin{algorithm}
\caption{Minimization on a subtree of the Stern-Brocot tree. (The final value of $\mD$.)}
\label{algo:MinTree}
\begin{tabular}{l}
\textbf{Input:} a finite set $V \subset \Z^2$, and a map $\vp : V \to \mR$.\\
\textbf{Initialize} a mutable vertex $f \leftarrow (1,0)$, and a mutable list $G \leftarrow [(0,1)]$. Set also $\mD \leftarrow + \infty.$\\
\textbf{While} $G$ is non-empty \textbf{do}\\
\phantom{bla} Denote by $g:=G_1$ the first element of $G$, and set $e := f+g$. \\
\phantom{bla} Denote $n := {\rm length}(G)$, and introduce the list $L:=[f+g,\ G_1+G_2, \cdots, G_{n-1}+G_n]$.\\
\phantom{bla} \textbf{If} $e \in V$\\
\phantom{blabla} \textbf{then} prepend $e$ to $G$, and set $\mD \leftarrow \min\{\mD, \vp(e)\}$\\
\phantom{blabla} \textbf{else} remove $g$ from $G$ and set $f \leftarrow g$ 
\end{tabular}
\end{algorithm}

\begin{lemma}
\label{lem:LSum}
At each iteration of the {\rm While} loop in Algorithm \ref{algo:MinTree}, one actually has $L = [f \oplus g, G_1 \oplus G_2, \cdots, G_{n-1} \oplus G_n]$. 
\end{lemma}

\begin{proof}
In the first iteration $L = [(1,0) \oplus (0,1)]$. 
We proceed by induction on the iteration index. Assume that $L = [f \oplus g, G_1 \oplus G_2, \cdots, G_{n-1} \oplus G_n]$. 
Since $e = f \oplus g$, we have $f+e = f\oplus e$ and $e+g = e\oplus g$ by Proposition \ref{prop:SumChildren}. If $e \in V$, then at the next iteration $L' = [f\oplus e, e\oplus g, G_1\oplus G_2, \cdots, G_{n-1}\oplus G_n]$. On the other hand, if $e\notin V$ then at the next iteration $L' = [G_1 \oplus G_2, \cdots, G_{n-1} \oplus G_n]$.
\end{proof}


In the following any set $V \subset \Z^2$ is \emph{regarded as a graph}, which edges are those of $\mT$ having their endpoints in $V$, see Definition \ref{def:GraphT}. 
In particular the standard Stern-Brocot tree has vertices $\mT^+ := \{(a,b) \in \mT; \, a>0, b>0\}$, and is a complete infinite binary tree of root $(1,1)$. We say that a binary tree is \emph{proper} iff its nodes have either two children (internal nodes) or zero (leaves). 

\begin{definition}
Let $V \subset \mZ^2$ be a finite subtree of $\mT^+$ with root $(1,1)$. 
We denote by $V_*\subset \mZ^2$ the proper binary subtree of $\mT^+$ which set of internal nodes is $V$. Note that $\#(V_*) = 2 \#(V)+1$.
\end{definition}


\begin{proposition}
\label{prop:MinTreePlus}
Let $V \subset \mZ^2$ be a finite subtree of $\mT^+$ with root $(1,1)$.
Algorithm \ref{algo:MinTree} conducts a depth-first transversal of the tree $V_*$, and at termination $\mD = \min \{\vp(e); \, e \in V\}$.
\end{proposition}

\begin{proof}
Let $e = f \oplus g \in V_*$. If $e \in V$, then $e$ has two children in $V_*$, namely $f\oplus e$ and $e \oplus g$. If $e \in V_*\sm V$, then $e$ is a leaf of $V_*$. Inspection of the proof of Lemma \ref{lem:LSum} shows that the list $L$ is updated precisely as expected for a depth-first transversal of $V_*$, see Algorithm \ref{algo:DepthFirst}. Since the operation $\mD \leftarrow \min\{\mD, \vp(e)\}$ is performed only for elements of $V$, it evaluates the minimum of $\vp$ on $V$.
\end{proof}

We finally introduce a  slight generalization of Proposition \ref{prop:MinTreePlus}, so as to more closely fit the context of Algorithm \ref{algo:D2p}, defining $\cDu_\cV u(x)$. 

\begin{corollary}
\label{corol:NonConnected}
Consider a finite finite set $V_0 \subset \Z^2$, and a map $\vp : V_0 \to \R$. 
Applying Algorithm \ref{algo:MinTree} to $(V_0, \vp)$ yields at termination $\mD = \min\{\vp(e); \, e \in V\}$, where $V\subset \mZ^2$ denotes the connected component of $(1,1)$ in $V_0$. 
\end{corollary}

\begin{corollary}
\label{corol:MinTree}
Let $V_0\subset \Z^2$ be a stencil, and let $\vp : V_0 \to \R$ be an even function. 
Let $V\subset \mZ^2$ denote the connected components of $\pm (1,1)$ and $\pm (1,-1)$ in $V_0$.
Apply Algorithm \ref{algo:MinTree} to $(V_0, \vp)$ with the modified initialization $G \leftarrow [(0,1), (-1,0)]$. Then at termination  $\mD = \min\{ \vp(e); \, e \in V\}$.
\end{corollary}

\begin{proof}
The execution of Algorithm \ref{algo:MinTree} with the modified initialization can be decomposed in two parts. (I) Execution with the standard initialization, which by Corollary \ref{corol:NonConnected} computes $\mD_+ := \min\{ \vp(e); \, e \in V \cap \mT^+\}$. (II) Execution with the modified initialization $f \leftarrow (0,1)$, $G \leftarrow [(-1,0)]$, which similarly computes $\mD_- := \min\{ \vp(e); \, e \in V \cap \mT^-\}$, with $\mT^- := \{(a,b) \in \mT; \, a<0,\, b>0\}$. 
Eventually $\mD = \min \{ \mD_+ , \mD_-\}$, which is the minimum of $\vp$ on $V$ since $\vp$ is even and $V$ is symmetric w.r.t the origin by Definition \ref{def:Stencil}.
\end{proof}

\subsection{Increasing functions on trees, and the case of quadratic functions}
\label{sec:PropagationOnTrees}

The hierarchical MA-LBR operator $\cDu_\cV u(x)$ can be regarded, essentially, see Corollary \ref{corol:MinTree}, as an minimization over a subtree of the Stern-Brocot tree. In this section, we identify assumptions under which this pruning procedure is valid, i.e.\ it only drops useless branches where the minimum would not be found. 

\begin{definition}
\label{def:IncreasingOutOfT}
Let $B$ be a graph, let $A$ be a subset of its vertices, and let $\vp : B \to \R$. We say that $\vp$ is increasing outside of $A$ iff for each edge  $a \to b$ of the graph $B$ with $b \in B \sm A$, one has $\vp(a) \leq \vp(b)$.
\end{definition}

\begin{proposition}
\label{prop:IncreasingMin}
Let $B$ be a finite collection of finite trees, and let $A$ be a subset of $B$ containing the root of each tree. If $\vp : B \to \mR$ is increasing outside of $A$, then $\min_A \vp = \min_B \vp$.
\end{proposition}

\begin{proof}
Let $b$ be a minimizer of $\vp$ on $B$, with minimal (graph) distance from the root of its tree. Assume for contradiction that $b \notin A$. Then $b$ is not the root, hence there exists an edge $a \to b$ in the graph $B$. Then $\vp(a) \leq \vp(b)$ and $a$ is closer to the root, which is a contradiction.
\end{proof}

Given some fixed $u \in \mU$, $x \in X$, we introduce the function $\vp : \mT \to \mR_+$ defined by 
\begin{equation}
\label{eqdef:phi}
\vp(e) := 
\begin{cases}
\huefg & \text{if } e=f\oplus g,\\
+\infty & \text{otherwise, i.e.\ if } e \in \{(\pm 1, 0), (0, \pm 1)\}.
\end{cases}
\end{equation}
We show in the next proposition, under some assumptions, that $\vp$ is increasing in the sense of Definition \ref{def:IncreasingOutOfT} on some subsets of the graph $\mT$.
From this we deduce the equality of the (non-adaptive) MA-LBR operator $\cD_\cV$ \eqref{eqdef:DVux} associated to some small and large stencils.
\begin{definition}
Let $u \in \FX$, let $x \in X$, and let $e = f \oplus g$. We define 
\begin{equation}
\label{eqdef:Heux}
H_e u(x) := \Delta_e u(x) - \Delta_f u(x) - \Delta_g u(x).
\end{equation}
\end{definition}

\begin{proposition}
\label{prop:PhiInc}
Let $u \in \FX$, let $x \in X$, and let $\cU, \cV$ be families of stencils. Assume that $\cV(x) \subset \cU(x) \subset \cVu(x)$, and that: 
\begin{enumerate}[(A)]
\item 
$\Delta_e u(x)>0$ for each $e \in \cVu(x)$.
\item 
$H_e u(x) \geq 0$ for each $e \in \cVu(x) \sm \cU(x)$.
\end{enumerate}
Then $\cD_\cU u(x) = \cD_\cVu u(x)$.
\end{proposition}

\begin{proof}
Fix $u \in \FX$, $x \in X$, and consider $\vp$ defined by \eqref{eqdef:phi}. Denote $V := \cV(x)$, $U := \cU(x)$, $\overline V := \cVu(x)$. We regard $\overline V$ as a subgraph of $\mT$, by keeping all edges with endpoints in this set. 

We claim that the restriction of $\vp$ to $\overline V$ is increasing outside of $U$, in the sense of Definition \ref{def:IncreasingOutOfT}.
Indeed consider an edge $e \to e'$ of $\mT$, where $e' \in \overline V \sm U$. Introducing the decomposition $e = f \oplus g$, we note that $e' \in \{f \oplus e, e \oplus g\}$, and also that $e,f,g \in \cVu(x)$ by Lemma \ref{lem:ReversedHierarchy}. For each $\omega \in \{e,f,g,e'\}$ let $\delta_\omega := \Delta_\omega u(x)$, which is positive by (A).
Assuming without loss of generality that $e' = f \oplus e$ we obtain 
$\delta_{e'} \geq \delta_e + \delta_f$, by (B).
Hence, as announced, using Lemma \ref{lem:hDiff}
\begin{equation}
\vp(e) = h(\delta_e, \delta_f, \delta_g) \leq \delta_e \delta_f  = h(\delta_{e'}, \delta_e, \delta_f)  = \vp(e').
\end{equation}

Thus $\vp : \overline V \to \mR$ is increasing outside of $U$, and therefore $\cD_\cU u (x) = \min\{\vp(e);\, e \in U\} = \min \{ \vp(e);\, e \in \overline V\} = \cD_\cVu u(x)$ by Proposition \ref{prop:IncreasingMin} and Corollary \ref{corol:PositiveDifferences2}. This concludes the proof. 
\end{proof}

We focus in the rest of this section on the case of a quadratic function $u_M$, where $M \in S_2^+$ is fixed. We link the adaptive MA-LBR operator $\cDu_\cV u_M(x)$ with Selling's algorithm page \pageref{algo:Selling}. Since $\Delta_e u_M(x) = \<e,Me\> > 0$ for any $x \in X$, $e \neq 0$, assumption (A) of Proposition \ref{prop:PhiInc} is automatically satisfied. Regarding (B) we observe the simplification: if $e=f \oplus g$ 
\begin{equation}
\label{eq:HScal}
H_e u_M(x) = \<f,M g\>.
\end{equation}
We thus introduce
\begin{equation}
\mT_M := \{e \in \mT; \, e= f \oplus g, \, \<f, M g\> < 0\}.
\end{equation}
We shall use the identity: for any $f,g \in \mR^2$ 
\begin{equation}
\label{eq:ComplexModuleMultiplication}
\det(f,g)^2+\<f,g\>^2 = \|f\|^2\|g\|^2.
\end{equation}

\begin{lemma}
\label{lem:FiniteTree}
If $M$ is diagonal, then $\mT_M = \emptyset$. Otherwise $\mT_M = \{e_0, \cdots, e_n, -e_0, \cdots, -e_n\}$, for  some finite branch $e_0 \to e_1 \to \cdots \to e_n$ of $\mT$, with $e_0 \in \{(1,1), (-1,1)\}$.
\end{lemma}

\begin{proof}
Claim (symmetry): one has $e \in \mT_M$ iff $-e \in \mT_M$. Indeed if $e=f\oplus g$ then $-e = (-f) \oplus (-g)$, and $\< f, M g\> = \<(-f), M (-g)\>$. Claim (tree structure): for any edge $e \to e'$ of $\mT$, one has $e' \in \mT_M \Rightarrow e \in \mT_M$. Indeed write $e = f \oplus g$, so that $e' = f \oplus e$ (resp.\ or $e' = e \oplus g$). Then $\<f, M e\> = \<f, M g\> + \<f,M f\> \geq \<f, M g\>$ (resp.\ likewise $\<e, M g\> \geq \<f, M g\>$) as announced. Claim (single branch): if $e = f \oplus g$, then at most one of $f \oplus e$ and $e \oplus g$ belongs to $\mT_M$. Indeed $\<f, M e\>+\<e, M g\> = \<e, M e\> \geq 0$, hence at most one of these scalar products is negative.

In order to conclude the proof, it suffices to establish the finiteness of $\mT_M$. 
Let $e = f \oplus g$, let $\lambda$ denote the smallest eigenvalue of $M$, and let $\kappa(M) := \sqrt{\|M\| \|M^{-1}\|}$. Finiteness follows from the claim: if $\|e\| \geq 1+ \kappa(M)$, then $e \notin \mT_M$. Indeed
\begin{equation*}
\<f,M g\>^2 = \|f\|_M^2 \|g\|_M^2 - \det(M) \det(f,g)^2 \geq \lambda^2 \|f\|^2 \|g\|^2 - \det(M) > \lambda^2 \left[(\|e\|-1)^2 - \kappa(M)^2\right]
\end{equation*}
We applied \eqref{eq:ComplexModuleMultiplication} to $M^\frac 1 2 f$ and $M^\frac 1 2 g$ for the first identity, and used that $\det(f,g)=1$ for the following inequality. Last inequality used $\|f\|+\|g\| > \|e\|$, $\min \{\|f\|, \|g\|\} \geq 1$, hence $\|f\| \|g\| > \|e\|-1$, and $\det(M)/\lambda^2 = \kappa(M)^2$. 
Since $\|e\| \geq 1+\kappa(M)$ we have shown $\<f, M g \> \neq 0$.
Apply this observation to the family of matrices $M_t := (1-t) \Id +t M$, $t\in [0,1]$, which satisfy $\kappa(M_t) \leq \kappa(M)$. Proposition \ref{prop:Decomp} states that $\<f,M_0 g\> = \<f,g\> \geq 0$, hence $\<f, M_t g\> > 0$ for all $t \in [0,1]$, thus $\<f, M g \> > 0$ and therefore $e \notin \mT_M$ as announced. 
\end{proof}

In the following corollary, a superbase $(e_0,e_1,e_2)$ of $\mZ^2$ is said to be equivalent to the superbases $(\ve e_i, \ve e_j, \ve e_k)$, for any permutation $\{i,j,k\}$ of $\{0,1,2\}$, and any sign $\ve \in \{-1,1\}$.
\begin{corollary}
\label{corol:AncestorHScal}
Let $M \in S_2^+$. If $M$ is diagonal, then Selling's algorithm stops at the first iteration. Otherwise let $e_0 \to \cdots \to e_n$ be as in Lemma \ref{lem:FiniteTree}, write $e_i = f_i \oplus g_i$. Then Selling's algorithm, initialized with the superbase $(e_0, -f_0, -g_0)$, generates in its successive iterations superbases equivalent to $(e_i,-f_i,-g_i)$. It terminates at the $n$-th iteration.
\end{corollary}

\begin{proof}
Claim: the superbase $(e_n,-f_n,-g_n)$ is $M$-obtuse. Indeed, $\<(-f_n), M (-g_n)\> = \<f_n, M g_n\> < 0$ since $e_n \in \mT_n$. On the other hand $f_n \oplus e_n \notin \mT_M$ and $e_n \oplus g_n \notin \mT_M$, by Lemma \ref{lem:FiniteTree} and the structure of $\mT$, see Definition \ref{def:GraphT}. Hence $\<(-f_n), M e_n\> \leq 0$ and $\<-g_n, M e_n\> \leq 0$.

Proof by induction on the iteration count $i$, $0 \leq i \leq n$. Case $i=0$ holds by the choice of initialization. Induction: consider the superbase $(e_i,-f_i, -g_i)$ of the $i$-th iteration, for some $0 \leq i<n$. Assume that $e_{i+1} = f_i \oplus e_i$ (the case $e_{i+1} = e_i \oplus g_i$ is similar), which means that $\<f_i, M e_i\> < 0$. One has $\<(-f_i), M (-g_i)\> < 0$ since $e_i \in \mT_M$, $\<(-f_i), M e_i\> > 0$, and $\<(-g_i), M e_i\> = - \<e_i, M e_i\> + \<f_i, M e_i\> < 0$. Hence Selling's algorithm constructs for the next iteration the superbase $(e_i - (-f_i), -f_i, -e_i) = (e_i\oplus f_i, -f_i, -e_i) = (e_{i+1},-f_{i+1},-g_{i+1})$ as announced.
\end{proof}

The next proposition establishes our main result Theorem \ref{th:Hierarchy} in the special case of quadratic functions. It also shows that the pruning procedure defining the adaptive operator $\cDu_\cV$ is extremely well behaved, since it only explores (in addition to the basic stencil $\cV(x)$) a single branch of the Stern-Brocot tree, within $\mT_M$ like Selling's algorithm. 

\begin{proposition}
Let $M \in S_2^+$, and let $u := u_M$. Let $x \in X$, let $\cV$ be a family of stencils, and 
\begin{align}
\label{eq:AncestorsRefinement}
\cU(x) &:= \cV(x) \cup \{e \in \cV_\Omega(x); \, H_e u(x) < 0\} \\
\nonumber
&= \cV(x) \cup (\mT_M \cap \cV_\Omega(x)).
\end{align}
Then $\cDu_\cV u(x) = \cD_\cU u(x) = \cD_\cVu u(x)$. 
In addition, when computing $\cDu_\cV u (x)$ through Algorithm \ref{algo:D2p}, the evaluation of \eqref{eqdef:phi} is performed only when $e \in \cU(x)$.
\end{proposition}

\begin{proof}
Let $U := \cU(x)$, $V := \cV(x)$, $\overline V := \cVu(x)$, so that $U = V \cup (\mT_M \cap \overline V)$ by \eqref{eq:HScal}. Note that the collection of subtrees of a same tree, with the same root, is stable by unions and intersections.
Hence $U$ satisfies property (Hierarchy) of stencils, by Lemma \ref{lem:FiniteTree} and Proposition \ref{prop:VExtended},

We recognize in \eqref{eq:AncestorsRefinement} the (Refinement test)
appearing in Algorithm \ref{algo:D2p}. Corollary \ref{corol:MinTree} applied to $\cU(x)$ and \eqref{eqdef:phi} states that $\cDu_\cV u(x) = \cD_\cU u(x)$. On the other hand $\cD_\cU u(x) = \cD_\cVu(x)$ by Proposition \ref{prop:PhiInc}, which concludes the proof. 
\end{proof}

\subsection{Equality of the adaptive and the extensive MA-LBR operator}
\label{sec:HierarchyConclusion}

We prove Theorem \ref{th:Hierarchy}, stating under mild assumptions the equality of the adaptive MA-LBR operator $\cDu_\cV$, and the brute-force one $\cD_\cVu$, which extensively sweeps through the extended stencils. 
For that purpose, and similarly to the quadratic case, we use through Proposition \ref{prop:PhiInc} the fact that the minimized function \eqref{eqdef:phi} is increasing on some portion of the Stern-Brocot tree. 

The key of the proof is the next proposition, preceded with a technical lemma, which weakens the assumptions of Proposition \ref{prop:PhiInc}. 
Strikingly, the stencils at each $x\in X$ cannot be dealt with independently. A simultaneous, and global argument is used instead, inspired by \cite{Mirebeau:Vn5Iu9VK}. 

\begin{figure}
\centering
\begin{tabular}{ccccc}
\includegraphics[height=2.5cm]{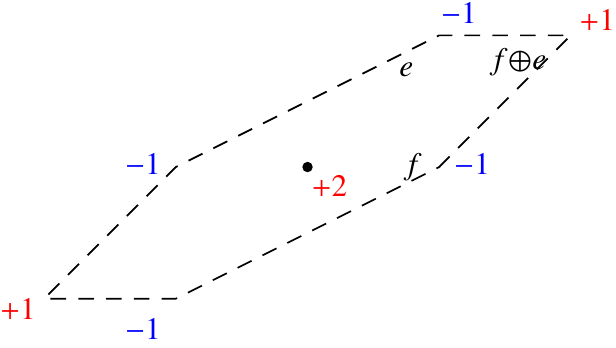} &
{\raise 1cm \hbox{\Large=}} &
\includegraphics[height=2.5cm]{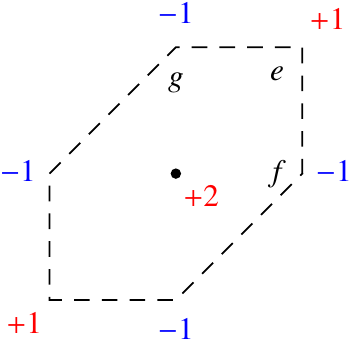} &
{\raise 1cm \hbox{\Large+}} &
\includegraphics[height=2.5cm]{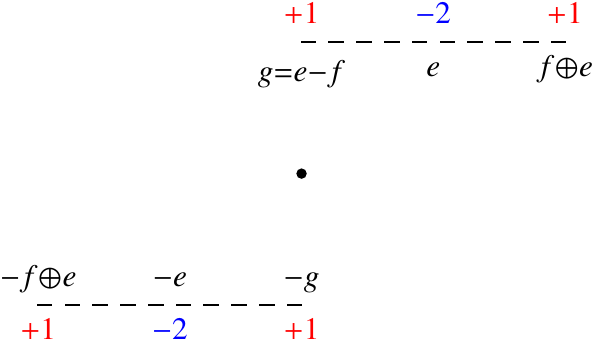} 
\end{tabular}
\caption{Illustration of Lemma \ref{lem:HHDD}}
\label{fig:HHDD}
\end{figure}

\begin{lemma}
\label{lem:HHDD}
Let $u \in \FX$, $x \in X$, and $e = f \oplus g $. If $e, e + f \in \cVin(x)$ then 
$(x+e)\pm f \in \Omega$, $(x-e) \pm f \in \Omega$, and
\begin{equation}
\label{eq:HHDD}
H_{e+f} u(x) = H_e u(x) + \Delta_f u(x+e) + \Delta_f u(x-e).
\end{equation}
Likewise if $e,e+g \in \cVin(x)$, exchanging the roles of $f$ and $g$.
\end{lemma}

\begin{proof}
Since $e \in \cVin(x)$ we have $x \pm e, x \pm f, x \pm g \in \Omega$. Since in addition $e + f \in \cVin(x)$ we have $x \pm (e+f) \in \Omega$. Note that $(x+e)-f = x+g$ and $(x-e)+f = x-g$.
Expanding the expressions on both sides of \eqref{eq:HHDD}, using that $e+f  = f \oplus e$ for the left side,
we find that they only involve the values of $u \in \FX$ at points of $X = \Omega \cap \Z^2$, and not on the boundary $\partial \Omega$. (As in \eqref{eqdef:Delta} and not \eqref{eqdef:DeltaBoundary}.) A cancellation occurs, as illustrated on Figure \ref{fig:HHDD}, and the result is proved.
\end{proof}

\begin{proposition}
\label{prop:DifferencesPropagation}
Let $\cU, \cV$ be families of stencils, and let $u \in \FX$. For each $x \in X$, assume that $\cV(x) \subset \cU(x) \subset \cVu(x)$, and that: 
\begin{enumerate}[(a)]
\item 
$\Delta_e u(x)>0$ for each $e \in \cU(x)$.
\item 
$H_e u(x) \geq 0$ for each $e \in \cVu(x) \sm \cU(x)$ for which there exists $e' \in \cU(x)$ such that $e' \to e$.
\end{enumerate}
Then $u, \cU, \cV$ satisfy the assumptions of Proposition \ref{prop:PhiInc}, for each $x \in X$.
\end{proposition}

\begin{proof}
Fix the stencils $\cV$, and proceed by decreasing induction on the cardinality $\#(\cU) := \sum_{x \in X} \#(\cU(x))$. If $\#(\cU) = \#(\cVu)$, then $\cU = \cVu$ and there is nothing to prove.

Assume that $\#(\cU) < \#(\cV)$, and consider a point $x \in X$ and a vector $e \in \cVu(x) \sm \cU(x)$, such that $\|e\|$ is minimal. Let us introduce the sets $\cU'(x) := \cU(x) \cup \{\pm e\}$, and $\cU'(y) := \cU(y)$ for all $y \neq x$, and note that $\#(\cU') = \#(\cU)+2$. We prove in the following that $\cU'$ is a family of stencils satisfying the assumptions (a) and (b). Hence by induction $\cU'$ satisfies (A) and (B), which immediately implies the same properties for $\cU$ and concludes the proof.

Proof that $\cU'$ is a family of stencils. Only (Hierarchy) needs to be checked.
Since $e \in \cVu(x) \sm \cU(x)$ we have $e \in \mT \sm V_8$, hence we may introduce the decomposition $e = f \oplus g$. By Lemma \ref{lem:DecompEdge}, either $f \to e$ or $g \to e$ is an edge of the graph $\mT$. By Lemma \ref{lem:ReversedHierarchy} we have $f,g \in \cVu(x)$. By minimality of $\|e\|$ we have $f,g \in \cU(x)$. 
By (Hierarchy) for $\cU$ the set $V_8$ has an element in the connected component of $f$ and $g$ in $\cU(x)$, hence in the connected component of $e$ in $\cU'(x)$.
This establishes (Hierarchy) for $\cU'$.

Proof that $\cU'$ satisfies (a). It suffices to check this property for the additional elements $\pm e$. Using (a) for $\cU$ we obtain $\Delta_f u(x) > 0$, $\Delta_g u(x)>0$. Using (b) for $\cU$ we get  $H_e u(x) \geq 0$.
Therefore
\begin{equation}
\Delta_e u(x) = H_e u(x) + \Delta_f u(x) + \Delta_g u(x) >0.
\end{equation}

Proof that $\cU'$ satisfies (b). The two edges originating from $e$ in the graph $\mT$ are $e \to f \oplus e$ and $e \to e \oplus g$, see Definition \ref{def:GraphT}. 
Let us assume that $e+f \in \cVu(x) \sm \cU'(x)$, and establish that $H_{e+f} u(x) \geq 0$. Note that $\cVu(x)\sm \cU(x) \subset (\cV(x) \cup \cVin(x)) \sm \cV(x) \subset \cVin(x)$, hence $e,e+f \in \cVin(x)$. 
Applying Lemma \ref{lem:AllSensibleDirections} we obtain $(x+e)\pm f \in \Omega$, hence $f \in \cVu(x+e)$ by Lemma \ref{lem:AllSensibleDirections}, thus $f \in \cU(x+e)$ by minimality of $\|e\|$, and therefore $\Delta_f u(x+e) > 0$ by (a) for the stencils $\cU$. Likewise $\Delta_f u(x-e) > 0$.
Using \eqref{eq:HHDD} yields as announced $H_{e+f} u(x) > 0$. Likewise $H_{e+g} u(x) > 0$ if $e+g \in \cVu(x) \sm \cU'(x)$. This establishes (b) for $\cU'$, and concludes the proof.
\end{proof}


Our last proposition immediately implies the announced Theorem \ref{th:Hierarchy}.

\begin{proposition}
\label{prop:Hierarchy}
Let $\cV$ be a family of stencils, and let $u \in \FX$. If $\cDu_\cV u > 0$  on $X$, then $\cDu_\cV u = \cD_\cVu u$ on $X$. In all cases $\cD_\cVu u \leq \cDu_\cV u$ on $X$.
\end{proposition}

\begin{proof}
Let $\cV$ be a family of stencils, and let $u \in \FX$. We introduce, for each $x \in X$ the set 
\begin{equation}
\label{eqdef:U0x}
\cU_0(x) := \cV(x) \cup \{e \in \cVin(x); \, H_e u(x) < 0\}.
\end{equation}
By construction $\cV(x) \subset \cU_0(x) \subset \cVu(x)$. 
We regard $\cU_0(x)$ as a subgraph of $\mT$, keeping all edges which endpoints are both in $\cU_0(x)$.
Denote by $\cU(x)$ the union of connected components intersecting $V_8$ in $\cU_0(x)$. By construction, $\cV(x) \subset \cU(x) \subset \overline \cV(x)$, and $\cU = (\cU(x))_{x \in X}$ is a family of stencils.

We recognize in the definition \eqref{eqdef:U0x} of $\cU_0(x)$ the (Refinement test) 
appearing in the computation of $\cDu_\cV(x)$ Algorithm \ref{algo:D2p}. Corollary \ref{corol:MinTree} applied to $\cU_0(x)$ and the map \eqref{eqdef:phi} thus states that: for any $x \in X$
\begin{equation}
\cDu_\cV u(x) = \cD_\cU u(x).
\end{equation}
Recalling that $\cU(x) \subset \overline \cV(x)$, for all $x \in X$, we obtain $\cD_\cVu u \leq \cD_\cU u = \cDu_\cV u$ on $X$ as announced. 

The stencils $\cU$ satisfy by construction assumption (b) of Proposition \ref{prop:DifferencesPropagation}.
Introducing the assumption that $\cDu_\cV = \cD_\cU$ is positive on $X$, and using 
Proposition \ref{prop:PositiveDifferences}, we find that $\cU$ also satisfies assumption (a) of Proposition \ref{prop:DifferencesPropagation}. Thus $\cD_\cU u = \cD_\cVu u$ on $X$, by Proposition \ref{prop:PhiInc}, which concludes the proof.
\end{proof}

%% file: OTBasis_NumericalExperiments.tex
\section{Numerical experiments}
\label{sec:Num}

We compare the introduced MA-LBR (Monge-Ampere using Lattice Basis Reduction), with two alternative solvers of Monge-Ampere equations. The Finite Differences scheme $\cD^\FD$, see \eqref{eqdef:DFD} and \cite{Loeper:2005fn}, is consistent but lacks the convergence guarantees associated to degenerate elliptic schemes. The Wide Stencil scheme $\cD^\FO_\cV$, see \eqref{eq:FO} and \cite{Froese:2013ez}, provides these guarantees, but at the price of a difficult compromise between consistency error and scheme locality, governed by the chosen stencil angular resolution see Figures \ref{fig:Stencils} and \ref{fig:RelativeError}.
Our numerical scheme the MA-LBR aims to combine the qualities of these two methods: consistency and monotony, with a comparable numerical cost. We use the MA-LBR adaptive implementation $\cDu_\cV$ of Algorithm \ref{algo:D2p}, with an $8$ points stencil $\cV(x) = V_8$, except on a layer of $4$ pixels along the domain boundary (where hierarchical refinement is mostly ineffective) where we use the $48$ points stencil of Figure \ref{fig:Stencils} (right).
The filtered scheme introduced in \cite{Froese:2013ez} also attempts to combine the strengths of the Wide Stencil scheme $\cD^\FO_\cV$ and the Finite Differences scheme $\cD^\FD$; this scheme is omitted in our experiments because it depends on several parameters, which make benchmarks and comparisons difficult.

We limit our attention to synthetic test cases, posed on the unit square $\Omega := ]0,1[^2$. A known convex function $U: \overline \Omega \to \R$ is numerically recovered from its hessian determinant $\rho := \det(\nabla^2 U)$, and its boundary values $\sigma := U_{|\partial \Omega}$. 
The tests are (supposedly) ordered by increasing difficulty, starting from a simple quadratic function and ending with a non-differentiable function (on a domain corner). 
\begin{itemize}
\item (Quadratic) $U(x) := \frac 1 2 \<x, M x\>$, where $M = M(\kappa, \theta)$ is as in \eqref{eq:MatParam} with  $\kappa:=10$, $\theta:=\pi/3$.
\item (Smoothed cone) $U(x) := \sqrt{\delta^2+ \|x-x_0\|^2}$, with $\delta := 0.1$ and $x_0:=(1/2,1/2)$. 
\item (Flat, \cite{Froese:2013ez})  $U(x) := (\|x-x_0\|_+ - r_0)^2 + \frac \ve 2 \|x-x_0\|^2$, with $r_0:=0.2$ and $\ve=10^{-6}$. 
\item (Singular, \cite{Froese:2013ez})  $U(x) := -\sqrt{2-\|x\|^2}$.
\end{itemize}
An iterative solver is applied to the discrete system \eqref{eq:DiscreteSys}, starting from a strictly convex seed, see Remark \ref{rem:Initialization}. Although the convergence guarantees of DE schemes only encompass Euler iterative solvers, we used without trouble a damped\footnote{%
Precisely, the iteration at a point $u \in \FX$ takes the form $u'= u+\delta^k v$, where $v$ is Newton's descent direction, $\delta:=0.7$, and $k \geq 0$ is the smallest integer such that: $\cD(u+\delta^k v)$ is positive on $X$ (except for scheme FD), and $\|f-\cD(u+\delta^k v)\|_{L^\infty(X)}$ is a local minimum in $k$. Convergence is numerically observed but not claimed in general.%
}
Newton solver. 
This may come as a surprise to those who regard Newton methods as local and excessively sensitive to initialization. The Monge-Ampere PDE fortunately benefits from a more favorable situation, since a suitably damped Newton method has been shown \cite{Loeper:2005fn} to converge globally - in the continuous setting, with periodic boundary conditions, and a Holder smooth positive right hand side. Discrete MA schemes which preserve the operator ellipticity may heuristically be expected to inherit this good behavior.

\begin{figure}
\centering
\includegraphics[width=5cm]{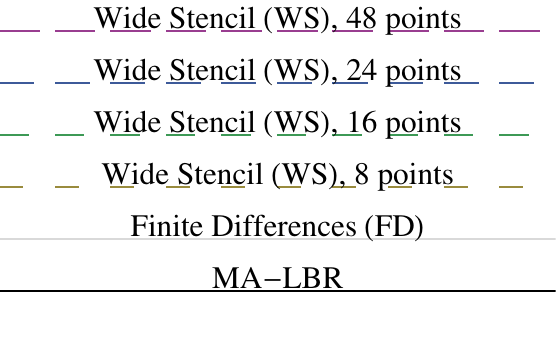}
\includegraphics[width=5cm]{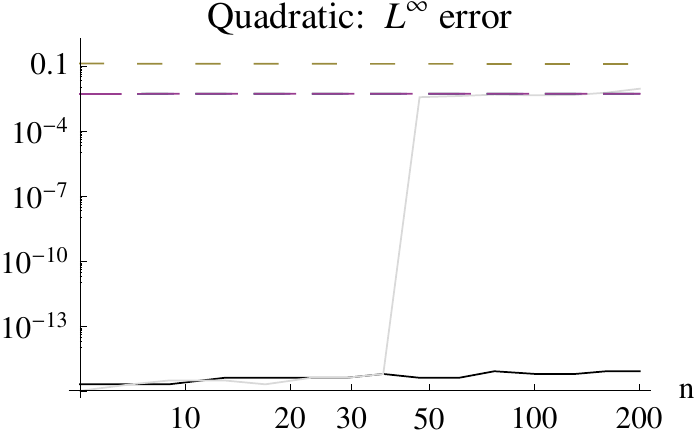}
\includegraphics[width=5cm]{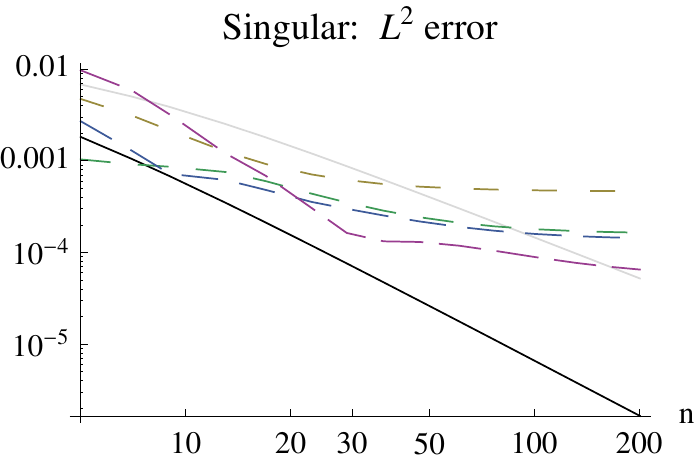}
\caption{Numerical error in the Quadratic and Singular cases. Logarithmic scale on all axes.}
\label{fig:Legend}
\label{fig:Quadratic}
\label{fig:SingularL1}
\end{figure}

Quadratic test case. The MA-LBR recovers this solution exactly, up to floating point errors, thanks to the adaptivity of  Algorithm \ref{algo:D2p} which refines the initial $8$ point stencil until the vector $(2,3)$ is included, and thus also the $M$-obtuse superbase $(2,3),(-1,-1),(-1,-2)$. Scheme FD also recovers the exact solution for a range of resolutions, but afterwards the discrete iterative solver switches to some erroneous alternative solution, see Figure \ref{fig:Quadratic}. Scheme WS produces a substantial $L^\infty$ error, which does not decrease with the grid scale. Indeed, it reflects a consistency error, and not a discretization error. 
Scheme FD could presumably recover the exact solution at all resolutions if its iterative solver was initialized more sensibly, for instance using the output of Scheme WS, or using a filtered combination of the two \cite{Froese:2013ez}.

\begin{figure}
\centering
\begin{tabular}{ccc}
\includegraphics[width=5cm]{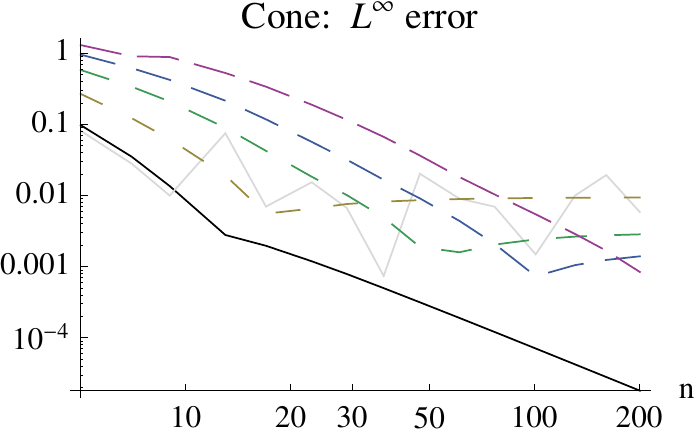} &
\includegraphics[width=5cm]{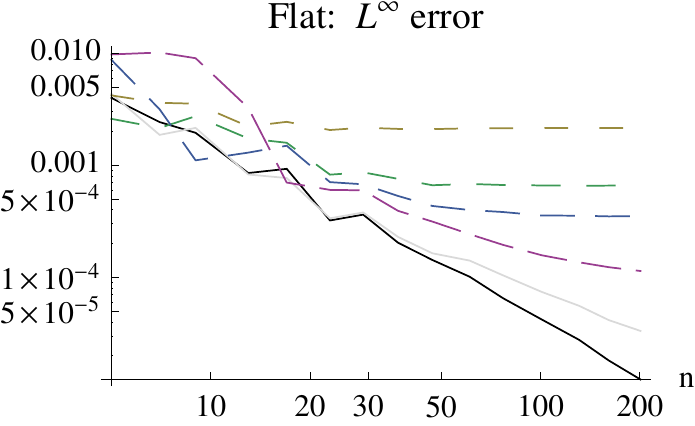} &
\includegraphics[width=5cm]{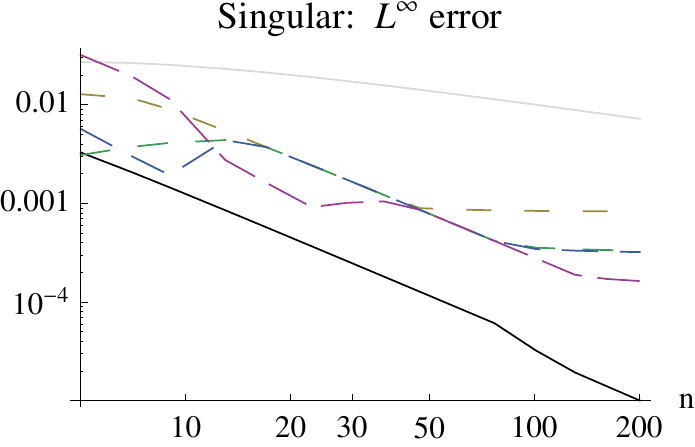} \\
\includegraphics[width=5cm]{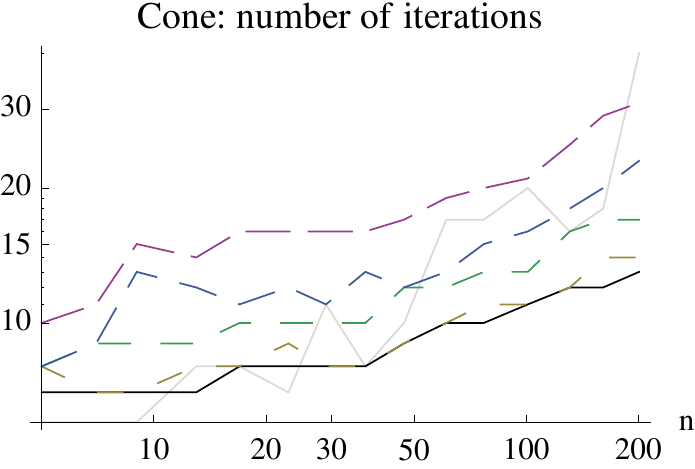} &
\includegraphics[width=5cm]{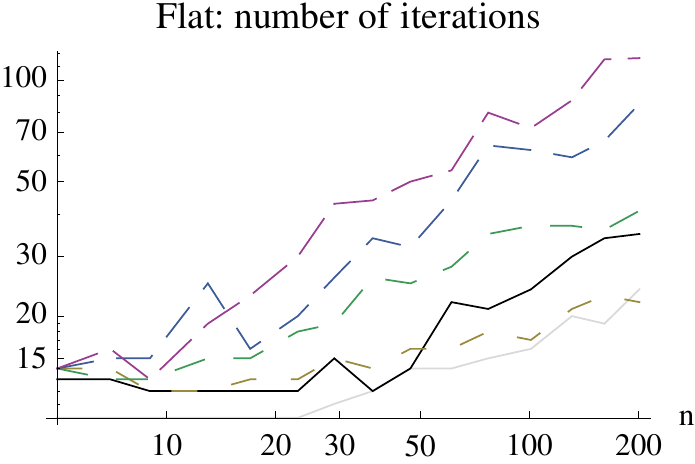} &
\includegraphics[width=5cm]{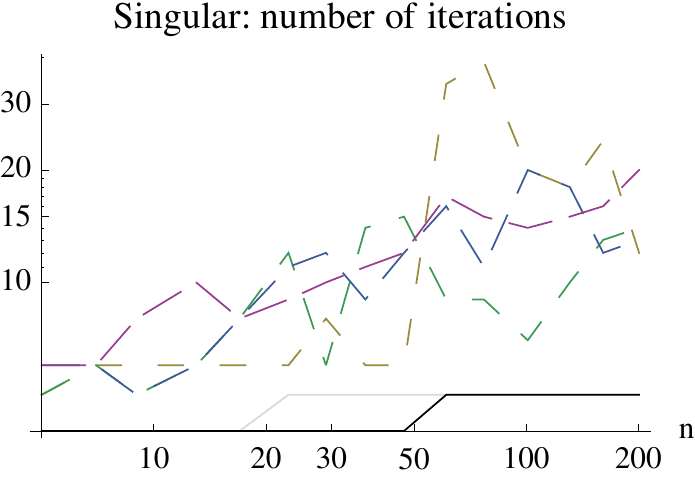} 
\end{tabular}
\caption{Numerical error and convergence speed in various test cases. Legend on Figure \ref{fig:Legend}.}
\end{figure}

Smoothed cone test case. The recovered function is $C^\infty$, yet its hessian is simultaneously (i) almost singular close to the domain boundary, and (ii) strongly peaked in a small region around the center. Scheme FD entirely fails this test. Choosing the best stencil for scheme WS is non-trivial, since point (i) suggests to use a large stencil for better angular resolution, but point (ii) mandates a scheme as local as possible. As a result, the best stencil, in terms of resulting $L^\infty$ error, successively has $8$, $16$, $24$ and $48$ points for grid sizes $n \times n$ with $n$ in the interval $[5,30]$, $[30, 80]$, $[80,160]$, $[160, \infty]$. The MA-LBR avoids the need for such manual parameter optimization, and produces numerical errors often one order of magnitude smaller. It also needs the least damped Newton iterations to reach convergence.

Flat test case. The recovered function is $C^1$, has a Lipschitz gradient, but is not $C^2$. It is also (almost) identically $0$ on a disk, up to a quadratic perturbation introduced to help the Newton solver. The effect of this perturbation on the numerical solution is negligible in comparison with the discretization error. The best stencil for Scheme WS is the largest one, with 48 points, for all resolutions $n \times n$ with $n \geq 17$. Despite the lack of $C^2$ regularity, scheme FD performs well in this test, better in fact than WS. The MA-LBR again outperforms the tested alternatives, and seems to provide a (slightly) improved asymptotic convergence rate in comparison with FD.

Singular test case. The recovered function is non differentiable at the domain corner $(1,1)$, where its gradient is formally $(+\infty, +\infty)$. Scheme FD fails this test, even if helped by initializing the iterative solver with a sampling of the known exact solution \cite{Froese:2011ka}. Regarding scheme WS, the $L^\infty$ error curves and the number of Newton iterations exhibit a puzzling erratic behavior: despite the scheme degenerate ellipticity, nasty things seem to occur close to singular point $(1,1)$. The $L^2$ error curve is smoother, see Figure \ref{fig:SingularL1}, and suggests that the optimal stencil size is successively $16$, $24$, $48$ at resolutions $n \times n$ with $n$ in the respective intervals $[0,8]$,$[8,23]$ and $[23, \infty]$ (note that an even larger stencil would be preferable at resolutions $\geq 100$). The MA-LBR avoids this difficult choice of stencil, and improves numerical error often by an order of magnitude. Our discretization handles well local singularities, and offers second order accuracy in smooth regions. The MA-LBR good balance is confirmed by the fast convergence of the damped Newton solver, which here never needs more than $5$ iterations.

\begin{remark}[Initialization]
\label{rem:Initialization}
We initialize the damped Newton iterative solver with the restriction $u = V_{|X \cup \partial \Omega} \in \FX$ of a strictly convex function $V \in C^0( \overline \Omega ) $, built using solely the  prescribed boundary conditions $\sigma$ on $\partial \Omega$.
 The construction is as follows:
(i) Find $\ve>0$ such that $\sigma_\ve(x) := \sigma(x)-\ve \|x\|^2$ is convex on any segment of $\partial \Omega$.
(ii) Find the maximal convex extension $\Sigma_\ve : \overline \Omega \to X$ of $\sigma_\ve$. 
This step requires the computation of a three dimensional convex hull, which is a classical problem of discrete geometry for which efficient procedures are available \cite{PrincetonUniversityDeptofComputerScience:1991uv}.
(iii) Initialize with the strictly convex $V(x) := \Sigma_\ve(x) + \ve \|x\|^2$. 
\end{remark}